\documentclass[a4paper,12pt]{article}

\usepackage{float}
\usepackage{amsmath}
\usepackage{amsthm}
\usepackage{amsfonts}
\usepackage{amssymb}
\usepackage[T1]{fontenc}
\usepackage{geometry}
\usepackage{graphicx}
\usepackage{hyperref}
\usepackage{icomma}
\usepackage[latin1]{inputenc}
\usepackage{latexsym}
\usepackage[numbers]{natbib}
\usepackage{textcomp}
\usepackage[all]{xy}
\usepackage{cancel}
\usepackage{color}
\usepackage{dsfont}

\title{Interacting particle systems and Yaglom limit approximation of diffusions with unbounded drift}
\author{Denis Villemonais\thanks{villemonais@cmap.polytechnique.fr,\newline \indent CMAP, \'Ecole Polytechnique, route de Saclay, 91128 Palaiseau, France}}

\DeclareMathSymbol{\minus}{\mathord}{operators}{"2D}

\newtheorem{theoreme}{Theorem}[section]

\newtheorem{proposition}[theoreme]{Proposition}
\newtheorem{remarque}{Remark}

\newtheorem{hypothese}{Hypothesis}

\begin{document}
\maketitle

\begin{abstract}

 We study the existence and the exponential ergodicity of a general interacting particle system, whose components are driven by independent diffusion processes with values in an open subset of $\mathds{R}^d$, $d\geq 1$. The interaction occurs when a particle hits the boundary: it jumps to a position chosen with respect to a probability measure depending on the position of the whole system.

Then we study the behavior of such a system when the number of particles goes to infinity. This leads us to an approximation method for the Yaglom limit of multi-dimensional diffusion processes with unbounded drift defined on an unbounded open set.
While most of known results on such limits are obtained by spectral theory arguments and are concerned with existence and uniqueness problems, our approximation method allows us to get numerical values
of quasi-stationary distributions, which find applications to many disciplines.
We end the paper with numerical illustrations of our approximation method for stochastic processes related to biological population models.

\end{abstract}

\noindent\textit{Key words : }diffusion process, interacting particle system, empirical process, quasi-stationary distribution, Yaglom limit.

\noindent\textit{MSC 2000 subject : }Primary 82C22, 65C50, 60K35; secondary 60J60

\section{Introduction}
\label{part1}

Let $D\subset\mathds{R}^d$ be an open set with a regular boundary (see Hypothesis \ref{hypD}). The first part of this paper is devoted to the study of interacting particle systems $(X^1,...,X^N)$, whose components $X^i$ evolve in $D$ as diffusion processes and jump when they hit the boundary $\partial D$.
More precisely, let $N\geq 2$ be the number of particles in our system. Let us consider $N$ independent $d$-dimensional Brownian motions $B^1,...,B^N$ and a jump measure ${\cal J}^{(N)}:\partial (D^N)\mapsto {\cal M}_1(D^N)$, where ${\cal M}_1(D^N)$ denotes the set of probability measures on $D^N$.
We build the interacting particle system $(X^1,...,X^N)$ with values in $D^N$ as follows. At the beginning, the particles $X^i$ evolve as independent diffusion processes with values in $D$ defined by
\begin{equation}
\label{int_e1}
 dX^{(i)}_t=dB^i_t+q^{(N)}_i(X^{(i)}_t)dt,\ X^{(i)}_0\in D,
\end{equation}
where $q^{(N)}_i$ is locally Lipschitz on $D$, such that the diffusion process doesn't explode in finite time. When a particle hits the boundary, say at time $\tau_1$, it jumps to a position chosen with respect to ${\cal J}^{(N)}(X^1_{\tau_1\minus},...,X^N_{\tau_n\minus})$. Then the particles evolve independently with respect to \eqref{int_e1} until one of them hits the boundary and so on. 
In the whole study, we require the jumping particle to be attracted away from the boundary by the other ones during the jump (in the sense of Hypothesis \ref{hyp1} on ${\cal J}^{(N)}$ in Section \ref{partA_1}). We emphasize the fact that the diffusion processes which drive the particles between the jumps can depend on the particles and their coefficients aren't necessarily bounded (see Hypothesis \ref{hypD}).
This construction is a generalization of the Fleming-Viot type model introduced in \cite{Burdzy2000} for Brownian particles and in \cite{Grigorescu2009a} for diffusion particles. Diffusions with jumps from the boundary have also been studied in \cite{Ben-Ari2009}, with a continuity condition on ${\cal J}^{(N)}$ that isn't required in our case, and in \cite{Grigorescu2007}, where fine properties of a Brownian motion with rebirth have been established.

In a first step, we show that the interacting particle system is well defined, which means that accumulation of jumps doesn't occur before the interacting particles system goes to infinity. Under additional conditions on $q_i^{(N)}$ and $D$, we prove that the interacting particle system doesn't reach infinity in finite time almost surely. 
In a second step, we give suitable conditions ensuring the system to be exponentially ergodic. The whole study is made possible thanks to a coupling between $(X^1,...,X^N)$ and a system of $N$ independent $1$-dimensional reflected diffusion processes. The coupling is built in Section \ref{partA_2}.

Assume that $D$ is bounded. For all $N\geq 2$, let ${\cal J}^{(N)}$ be a jump measure and $(q_i^{(N)})_{1\leq i\leq N}$ a family of drifts.  Assume that the conditions for existence and ergodicity of the interacting process are fulfilled for all $N\geq 2$. Let $M^N$ be its stationary distribution. We denote by ${\cal X}^N$ the associated empirical stationary distribution, which is defined by ${\cal X}^N=\frac{1}{N}\sum_{i=1}^N\delta_{x_i}$, where $(x_1,...,x_N)\in D^N$ is distributed following $M^N$. Under some bound assumptions on $(q_i^{(N)})_{1\leq i\leq N, 2\leq N}$ (see Hypothesis \ref{hypQ}), we prove in Section \ref{partA_3} that the family of random measures ${\cal X}^N$ is uniformly tight.

In Section \ref{partApproximation}, we study a particular case: $q^{(N)}_i=q$ doesn't depend on $i,N$ and 
\begin{equation}
 \label{int_e2}
 {\cal J}^{(N)}(x_1,...,x_N)=\frac{1}{N-1}\sum_{j\neq i}\delta_{x_j},\ x_i\in\partial D.
\end{equation}
It means that at each jump time, the jumping particle is sent to the position of a particle chosen uniformly between the $N-1$ remaining ones. In this situation, we identify the limit of the family of empirical stationary distributions $({\cal X}^N)_{N\geq 2}$.
This leads us to an approximation method of limiting conditional distributions of diffusion processes absorbed at the boundary of an open set of $\mathds{R}^d$, studied by Cattiaux and M\'el\'eard in \cite{Cattiaux2008} and defined as follows. Let $U_{\infty}\subset \mathds{R}^d$ be an open set and $\mathds{P}^{\infty}$ be the law of the diffusion process defined by the SDE
\begin{equation}
\label{int_e3}
 dX^{\infty}_t=dB_t+\nabla V(X^{\infty}_t)dt,\ X^{\infty}\in U_{\infty}
\end{equation}
and absorbed at the boundary $\partial U_{\infty}$. Here $B$ is a $d$-dimensional Brownian motion and $V\in C^2(U_{\infty},\mathds{R})$. We denote by $\tau_{\partial}$ the absorption time of the diffusion process \eqref{int_e3}. As proved in \cite{Cattiaux2008}, the limiting conditional distribution
\begin{equation}
 \label{intyag1}
 \nu_{\infty}=\lim_{t\rightarrow\infty}\mathds{P}^{\infty}_x\left( X^{\infty}_t\in . | t<\tau_{\partial}\right)
\end{equation}
exists and doesn't depend on $x\in U_{\infty}$, under suitable conditions which allow the drift $\nabla V$ and the set $U_{\infty}$ to not fulfill the conditions of Section \ref{partA} (see Hypothesis \ref{H3} in Section \ref{partApproximation}). This probability is called the Yaglom limit associated with $\mathds{P}^{\infty}$. It is a quasi-stationary distribution for the diffusion process \eqref{int_e3}, which means that $\mathds{P}^{\infty}_{\nu_{\infty}}(X^{\infty}_t\in dx| t<\tau_{\partial})=\nu_{\infty}$ for all $t\geq 0$. We refer to   \cite{Cattiaux2009, Kolb2010,Lladser2000} and references therein for existence or uniqueness results on quasi-stationary distributions in other settings.

Yaglom limits are an important tool in the theory of Markov processes with absorbing states, which are commonly used in stochastic models of
biological populations, epidemics, chemical reactions and market dynamics (see the bibliography \cite[Applications]{Polett}). Indeed, while the long time behavior of a recurrent Markov process is well described by its stationary distribution, the stationary distribution of an absorbed Markov process is concentrated on the absorbing states, which is of poor interest. In contrast, the limiting distribution of the process
conditioned to not being absorbed when it is observed can explain some complex behavior, as the mortality plateau at advanced ages (see \cite{Aalen2001} and \cite{Steinsaltz2004}), which leads to new applications of Markov processes with absorbing states in biology (see \cite{Li2009}). As stressed in \cite{NAaSELL2001}, such distributions are in most cases not explicitly computable. 
In \cite{Cattiaux2008}, the existence of the Yaglom limit is proved by spectral theory arguments, which doesn't allow us to get its explicit value. The main motivation of Section \ref{partApproximation} is to prove an approximation method of $\nu_{\infty}$, even when the drift $\nabla V$ and the domain $U_{\infty}$ don't fulfill the conditions of Section \ref{partA}.

The approximation method is based on a sequence of interacting particle systems defined with the jump measures \eqref{int_e2}, for all $N\geq 2$.
In the case of a Brownian motion absorbed  at the boundary of a bounded open set (\emph{i.e.} $q=0$), Burdzy \emph{et} \emph{al}. conjectured in \cite{Burdzy1996} that the unique limiting measure of the sequence $({\cal X}^N)_{N\in\mathds{N}}$ is the Yaglom limit $\nu_{\infty}$. This has been confirmed in the Brownian motion case (see \cite{Burdzy2000}, \cite{Grigorescu2004} and \cite{Lobus2009}) and proved in \cite{Ferrari2007} for some Markov processes defined on discrete spaces.
New difficulties arise from our case. For instance, the interacting particle process introduced above isn't necessarily well defined, since it doesn't fulfill the conditions of Section \ref{partA}. To avoid this difficulty, we introduce a cut-off of $U_{\infty}$ near its boundary.
More precisely, let $(U_m)_{m\geq 0}$ be an increasing family of regular bounded subsets of $U_{\infty}$, such that $\nabla V$ is bounded on each $U_{\infty}$ and such that $U_{\infty}=\bigcup_{m\geq 0} U_{\infty}$. We define an interacting particle process $(X^{m,1},...,X^{m,N})$ on each subset $U_m^N$, by setting $q^{(N)}_i=\nabla V$ and $D=U_{m}$ in \eqref{int_e1}.
 For all $m\geq 0$ and $N\geq 2$, $(X^{m,1},...,X^{m,N})$ is well defined and exponentially ergodic. Denoting by ${\cal X}^{m,N}$ its empirical stationary distribution, we prove that
\begin{equation*}
 \lim_{m\rightarrow\infty}\lim_{N\rightarrow\infty} {\cal X}^{m,N}=\nu_{\infty}.
\end{equation*}

We conclude in Section \ref{simulation} with some numerical illustrations of our method applied to the $1$-dimensional Wright-Fisher diffusion conditioned to be absorbed at $0$, to the Logistic Feller diffusion and to the $2$-dimensional stochastic Lotka-Volterra diffusion.

\section{A general interacting particle process with jumps from the boundary}
\label{partA}

\subsection{Construction of the interacting process}

Let $D$ be an open subset of $\mathds{R}^d$, $d\geq 1$.
Let $N\geq 2$ be fixed. For all $i\in\{1,...,N\}$, we denote by $\mathds{P}^i$ the law of the diffusion process $X^{(i)}$, which is defined on $D$ by
\begin{equation}
 \label{s11_e1b}
 dX^{(i)}_t=dB^i_t-q^{(N)}_i(X^{(i)}_t)dt,\ X^{(i)}_0=x^i\in D
\end{equation}
and is absorbed at the boundary $\partial D$. Here $B^1,...,B^N$ are $N$ independent $d$-dimensional Brownian motions and $q^{(N)}_i=(q^{(N)}_{i,1},...,q^{(N)}_{i,d})$ is locally Lipschitz. We assume that the process is absorbed in finite time almost surely and that it doesn't explode to infinity in finite time almost surely.

 The infinitesimal generator associated with the diffusion process \eqref{s11_e1b} will be denoted by ${\cal L}_i^{(N)}$, with
\begin{equation*}
 {\cal L}_i^{(N)}=\frac{1}{2}\sum_{j=1}^d{\frac{\partial^2}{\partial x_j^2}}-q^{(N)}_{i,j}\frac{\partial}{\partial x_j}
\end{equation*}
on its domain ${\cal D}_{{\cal L}_i^{(N)}}$.

For each $i\in\{1,...,N\}$, we set 
\begin{equation*}
\label{s11_e3}
{\cal D}_i=\{ (x_1,...,x_N)\in \partial (D^N),\ \text{such that}\ x_i\in\partial D,\ \text{and},\ \forall j\neq i,\ x_j\in D  \}.
\end{equation*}
We define a system of particles $(X^1,...,X^N)$ with values in $D^N$, which is c\`adl\`ag and whose components jump from $\bigcup_i {\cal D}_i$. Between the jumps, each particle evolves independently of the other ones with respect to $\mathds{P}^i$. 

Let ${\cal J}^{(N)}:\bigcup_{i=0}^N{\cal D}_i\rightarrow{\cal M}_1(D)$ be the jump measure, which associates a probability measure ${\cal J}^{(N)}(x_1,...,x_N)$ on $D$ to each point $(x_1,...,x_N)\in \bigcup_{i=1}^N{\cal D}_i$.
Let $(X^1_0,...,X^N_0)\in D^N$ be the starting point of the interacting particle process $(X^1,...,X^N)$, which is built as follows:
\begin{itemize}
 \item Each particle evolves following the SDE \eqref{s11_e1b}
 independently of the other ones, until one particle, say $X^{i_1}$, hits the boundary at a time which is denoted by $\tau_1$. 
On the one hand, we have $\tau_1>0$ almost surely, because each particle starts in $D$. On the other hand, the particle which hits the boundary at time $\tau_1$ is unique, because the particles evolves as independent It\^o's diffusion processes in $D$. It follows that $(X^1_{\tau_1\minus},...,X^N_{\tau_1\minus})$ belongs to ${\cal D}_{i_1}$.
 \item The position of $X^{i_1}$ at time $\tau_1$ is then chosen with respect to the probability measure ${\cal J}^{(N)}(X^1_{\tau_1\minus},...,X^N_{\tau_1\minus})$.
 \item At time $\tau_1$ and after proceeding to the jump, all the particles are in $D$. Then the particles evolve with respect to \eqref{s11_e1b} and independently of each other, until one of them, say $X^{i_2}$, hits the boundary, at a time which is denoted by $\tau_2$. As above, we have $\tau_1<\tau_2$ and $(X^1_{\tau_2\minus},...,X^N_{\tau_2\minus})\in{\cal D}_{i_2}$.
  \item The position of $X^{i_2}$ at time $\tau_2$ is then chosen with respect to the probability measure ${\cal J}^{(N)}(X^1_{\tau_2\minus},...,X^N_{\tau_2\minus})$.
  \item Then the particles evolve with law $\mathds{P}^{i}$ and independently of each other, and so on. 
\end{itemize}
The law of the interacting particle process with initial distribution $m\in{\cal M}_1(D^N)$ will be denoted by $P^{N}_m$, or by $P^{N}_x$ if $m=\delta_x$, with $x\in D^N$. The associated expectation will be denoted by $E^{N}_{m}$, or by $E_x$ if $m=\delta_x$.
For all $\beta>0$, we denote by $S_{\beta}=\inf\{t\geq 0,\ \|(X_1,...,X_N)\|\geq \beta\}$ the first exit time from $\{x\in D^N,\ \|x\|< \beta\}$. We set $S_{\infty}=\lim_{\beta\rightarrow\infty} S_{\beta}$.

The sequence of successive jumping particles is denoted by $(i_n)_{n\geq 1}$, and
\begin{equation*}
\label{s11_e5b}
0<\tau_1<\tau_2< ...
\end{equation*}
denotes the strictly increasing sequence of jumping times (which is well defined for all $n\geq 0$ since the process is supposed to be absorbed in finite time almost surely). Thanks to the non-explosion assumption on each $\mathds{P}^i$, we have $\tau_n<S_{\infty}$ for all $n\geq 1$ almost surely. We set $\tau_{\infty}=\lim_{n\rightarrow\infty}{\tau_n}\leq S_{\infty}$. The process described above isn't necessarily well defined for all $t\in[0,S_{\infty}[$, and we need more assumptions on $D$ and on the jump measure ${\cal J}^{(N)}$ to conclude that $\tau_{\infty}=S_{\infty}$ almost surely.

In the sequel, we denote by $\phi_D$ the Euclidean distance to the boundary $\partial D$:
\begin{equation*}
 \label{s11_e6}
 \phi_D(x)=\inf_{y\in\partial D} \|y-x\|_2,\ \text{for all}\ x\in D.
\end{equation*}
For all $r>0$, we define the collection of open subsets $D_r=\{x\in D,\ \phi_D(x)> r\}$. For all $\beta>0$, we set $B_{\beta}=\{x\in D, \|x\|<\beta\}$.

\begin{hypothese}
 \label{hypD}
There exists a neighborhood $U$ of $\partial D$ such that
 \begin{enumerate}
\item the distance $\phi_D$ is of class $C^2$ on U,
\item for all $\beta>0$,
\begin{equation*}
\inf_{x\in U\cap B_{\beta},\ i\in\{1,...,N\}} {\cal L}_i^{(N)} \phi_D(x) > -\infty.
\end{equation*}
\end{enumerate}
\end{hypothese}
In particular, Hypothesis \ref{hypD} implies
 \begin{equation}
 \label{s12_e4}
 \|\nabla \phi_D(x)\|_2=1,\ \forall x\in U.
 \end{equation}

\begin{remarque}\upshape
 For example, the first part of Hypothesis \ref{hypD} is fulfilled if $D$ is an open set whose boundary is of class $C^2$ (see \cite[Theorem 4.3]{Delfour2001}). It is also satisfied by the rectangle with rounded corner defined in Section \ref{partApproximation_3}.
\end{remarque}

The following assumption ensures that the jumping particle is attracted away from the boundary by the other ones.
 \begin{hypothese}
 \label{hyp1}
  There exists a non-decreasing continuous function $f^{(N)}:\mathds{R}_+\rightarrow\mathds{R}_+$ vanishing at $0$ and strictly increasing in a neighborhood of $0$ such that, $\forall i\in\{1,...,N\}$,
  \begin{equation*}
   \label{s11_e6b}
   \inf_{(x_1,...,x_N)\in{\cal D}_i}{\cal J}^{(N)}(x_1,...,x_N)(\{y\in D,\ \phi_D(y)\geq\min_{j\neq i}f^{(N)}(\phi_D(x_j)) \})\geq p_0^{(N)},
  \end{equation*}
$p_0^{(N)}>0$ is a positive constant.
 \end{hypothese}
Informally, $f^{(N)}(\phi_D)$ is a kind of distance from the boundary and we assume that at each jump time $\tau_n$, the probability of the event "the jump position $X^{i_n}_{\tau_n}$ is chosen farther from the boundary than at least one another particle" is bounded below by a positive constant $p_0^{(N)}$.

\begin{remarque}\upshape

Hypothesis \ref{hyp1} is very general and allows a lot of choices for ${\cal J}^{(N)}(x_1,...,x_N)$.
For instance, for all $\mu\in{\cal M}_1(D)$, one can find a compact set $K\subset D$ such that $\mu(K)>0$. Then ${\cal J}^{(N)}(x_1,...,x_N)=\mu$ fulfills the assumption with $p_0^{(N)}=\mu(K)$ and $f^{(N)}(\phi_D)=\phi_D\wedge d(K,\partial D)$.

Hypothesis \ref{hyp1} also includes the case studied by Grigorescu and Kang in \cite{Grigorescu2009a}, where
 \begin{equation*}
  \label{s11_e8}
  {\cal J}^{(N)}(x_1,...,x_N)= \sum_{j\neq i} p_{ij}(x_i) \delta_{x_j},\ \forall (x_1,...,x_N)\in {\cal D}_i.
 \end{equation*}
 with $\sum_{j\neq i} p_{ij}(x_i)=1$ and $\inf_{i\in\{1,...,N\},j\neq i, x_i\in \partial D} p_{ij}(x_i)>0$. 
In that case, the particle on the boundary jumps to one of the other ones, with positive weights. It yields that Hypothesis \ref{hyp1} is fulfilled with $p_0^{(N)}=1$ and $f^{(N)}(\phi_D)=\phi_D$.
In Section \ref{partApproximation}, we will focus on the particular case
 \begin{equation*}
  \label{s11_7}
  {\cal J}^{(N)}(x_1,...,x_N)=\frac{1}{N-1}\sum_{j=1,...,N,\ j\neq i}{\delta_{x_j}},\ \forall (x_1,...,x_N)\in{\cal D}_i.
 \end{equation*}
 That will lead us to an approximation method of the Yaglom limit \eqref{intyag1}.

Finally, given a jump measure ${\cal J}^{(N)}$ satisfying Hypothesis 2 (with $p_0^{(N)}$ and $f^{(N)}$), any $\sigma^{(N)}:\bigcup_{i=0}^N{\cal D}_i\rightarrow{\cal M}_1(D)$ and a constant $\alpha^{(N)}>0$, the jump measure
  \begin{equation*}
  {\cal J}_{\sigma}^{(N)}(x_1,...,x_N)=\alpha^{(N)}{\cal J}^{(N)}(x_1,...,x_N)+(1-\alpha^{(N)})\sigma^{(N)}(x_1,...,x_N),\ \forall (x_1,...,x_N)\in{\cal D}_i,
 \end{equation*}
 fulfills the Hypothesis \ref{hyp1} with $p_{0,\sigma}^{(N)}=\alpha^{(N)}p_0^{(N)}$ and $f^{(N)}_{\sigma}(\phi_D)=f^{(N)}(\phi_D)$.
\end{remarque}

Finally, we give a condition which ensures the exponential ergodicity of the process. In particular, this condition is satisfied if $D$ is bounded and fulfills Hypothesis \ref{hypD}.
\begin{hypothese}
\label{hypE}
There exists $\alpha>0$, $t_0^{(N)}>0$ and a compact set $K^{(N)}_0\subset D$ such that
 \begin{enumerate}
\item the distance $\phi_D$ is of class $C^2$ on $D\setminus D_{2\alpha}$ and
\begin{equation*}
\inf_{x\in D\setminus D_{2\alpha}, \ i\in\{1,...,N\}} {\cal L}_i^{(N)} \phi_D(x) > -\infty.
\end{equation*}
\item for all $i\in\{1,...,N\}$, we have
\begin{equation*}
p_1^{(N)}=\prod_{i=1}^N\inf_{x\in D_{\alpha/2}}\mathds{P}^{i}_x(X^{(i)}_{t^{(N)}_0}\in K^{(N)}_0)>0.
\end{equation*}
\end{enumerate}
\end{hypothese}

\begin{theoreme}
 \label{thA_1}
 Assume that Hypotheses \ref{hypD} and \ref{hyp1} are fulfilled. Then the process $(X^1,...,X^N)$ is well defined, which means that $\tau_{\infty}=S_{\infty}$ almost surely.

If Hypothesis \ref{hyp1} and the first point of Hypothesis \ref{hypE} are fulfilled, then $\tau_{\infty}=S_{\infty}=+\infty$ almost surely.

If Hypotheses \ref{hyp1} and \ref{hypE} are fulfilled, then the process $(X^1,...,X^N)$  is exponentially ergodic, which means that there exists a probability measure $M^N$ on $D^N$ such that,
\begin{equation*}
 \label{eq11}
 ||P^N_x((X^1_t,...,X^N_t)\in .)-M^N||_{TV}\leq C^{(N)}(x)\left(\rho^{(N)}\right)^t,\ \forall x\in D^N,\ \forall t\in\mathds{R}_+,
\end{equation*}
where $C^{(N)}(x)$ is finite, $\rho^{(N)}<1$ and $||.||_{TV}$ is the total variation norm. In particular, $M^N$ is a stationary measure for the process $(X^1,...,X^N)$.
\end{theoreme}

The main tool of the proof is a coupling between $(X^1_t,...,X^N_t)_{t\in [0,S_{\beta}]}$ and a system of
$N$ independent one-dimensional diffusion processes $(Y^{\beta,1}_t,...,Y^{\beta,N}_t)_{t\in[0,S_{\beta}]}$, for each $\beta>0$. The system is built in order to satisfy
\begin{equation*}
\label{s11_e10}
 0\leq Y^{\beta,i}_t \leq \phi_D(X^i_t)\ \text{a.s.}
\end{equation*}
for all $t\in [0,\tau_{\infty}\wedge S_{\beta}]$ and each $i\in\{1,...,N\}$. We build this coupling in Subsection \ref{partA_1} and we conclude the proof of Theorem \ref{thA_1} in Subsection \ref{partA_2} .

In Subsection \ref{partA_3}, we assume that $D$ is bounded and that, for all $N\geq 2$, we're given ${\cal J}^{(N)}$ and a family of drifts $(q_i^{(N)})_{1\leq i\leq N}$, such that Hypotheses \ref{hypD}, \ref{hyp1}  and \ref{hypE} are fulfilled. Moreover, we assume that $\alpha$ in Hypothesis \ref{hypE} doesn't depend on $N$.
Under some suitable bounds on the family $(q_i^{(N)})_{1\leq i\leq N,\ N\geq 2}$, we prove that the family of empirical distributions $({\cal X}^N)_{N\geq 2}$ is uniformly tight. It means that, $\forall \epsilon\geq 0$, there exists a compact set $K\subset D$ such that $E({\cal X}^N(D\setminus K))\leq \epsilon$ for all $N\geq 2$. In particular, this implies that $({\cal X}^N)_{N\geq 2}$ is weakly compact, thanks to \cite{Jakubowski1988}. Let us recall that a sequence of random measures $(\gamma_N)_N$ on $D$ converges weakly to a random measure $\gamma$ on $D$, if $E(\gamma_N(f))$ converges to $E(\gamma(f))$ for all continuous bounded functions $f:D\rightarrow \mathds{R}$. This property will be crucial in Section \ref{partApproximation}.

\subsection{Coupling's construction}
\label{partA_1}

\begin{proposition}
\label{le1}
Assume that Hypothesis \ref{hypD} is fulfilled and fix $\beta>0$. Then there exists $a>0$, a $N$-dimensional Brownian motion $(W^1,...,W^N)$ and positive constants $Q_1,...,Q_N$ such that, for each $i\in \{1,...,N\}$, the reflected diffusion process with values in $[0,a]$ defined by the reflection equation  (cf. \cite{Chaleyat-Maurel1978})
\begin{equation}
 \label{s12_e1}
 Y^{\beta,i}_t=Y^{\beta,i}_0+W^i_t-Q_i t+L_t^{i,0}-L_t^{i,a},\ Y^{\beta,i}_0=\min(a,\phi_D(X^i_0))
\end{equation}
satisfies
\begin{equation}
 \label{s12_e2}
 0\leq Y^{\beta,i}_t \leq \phi_D(X^i_t)\wedge a\ \text{a.s.}
\end{equation}
for all $t\in[0,\tau_{\infty}\wedge S_{\beta}[$ (see Figure \ref{fig2dc}). In \eqref{s12_e1},  $L^{i,0}$ (resp. $L^{i,a}$) denotes the local time of $Y^{\beta,i}$ at $\{0\}$ (resp. $\{a\}$).
\end{proposition}

\begin{remarque}\upshape
If the first part of Hypothesis \ref{hypE} is fulfilled, then the proof remains valid with $\beta=\infty$ and $a=\alpha$ (where $\alpha>0$ is defined in Hypothesis \ref{hypE}). This leads us to a coupling between $X^i$ and $Y^{\infty,i}$, valid for all $t\in[0,\tau_{\infty}\wedge S_{\infty}[=[0,\tau_{\infty}[$.
\end{remarque}

\begin{figure}[htbp]
\begin{center}
\includegraphics[height=70mm]{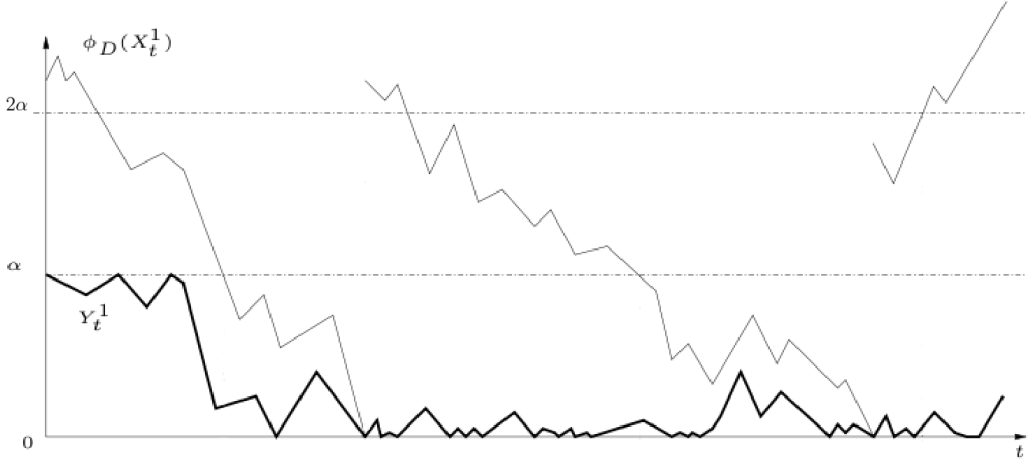}
\caption{The particle $X^1$ and its coupled reflected diffusion process $Y^1$}
\label{fig2dc}
\end{center}
\end{figure}

\begin{proof}[Proof of Proposition \ref{le1} :]
The set $\overline{B_{\beta}\setminus U}$ is a compact subset of $D$, then there exists $a>0$ such that $B_{\beta}\setminus U \subset D_{2a}$. In particular, we have $B_{\beta}\setminus D_{2a}\subset U$, so that $\phi_D$ is of class $C^2$ in $B_{\beta} \setminus D_{2a}$.

Fix $i\in\{1,...,N\}$.
 We define a sequence of stopping times $({\theta}^i_n)_n$ such that $X^i_t\in B_{\beta} \setminus D_{2 a}$ for all $t\in[{\theta}^i_{2n},{\theta}^i_{2n+1}[$ and $X^i_t\in \overline{D_{a}}$ for all $t\in[{\theta}^i_{2n+1},{\theta}^i_{2n+2}[$. More precisely, we set (see Figure \ref{fig2db})
 \begin{align*}
	{\theta}^i_0&=\inf{\{t\in[0,+\infty[,\ X^i_t\in B_{\beta} \setminus D_{a}\}}\wedge\tau_{\infty}\wedge S_{\beta},\\
	{\theta}^i_1&=\inf{\{t\in[t_0,+\infty[,\ X^i_t\in \overline{D_{2a}}\}}\wedge\tau_{\infty}\wedge S_{\beta},
\end{align*}
and, for $n\geq 1$,
\begin{align*}
	{\theta}^i_{2n}&=\inf{\{t\in[t^i_{2n-1},+\infty[,\ X^i_t\in B_{\beta} \setminus D_{a} \}}\wedge\tau_{\infty}\wedge S_{\beta},\\
	{\theta}^i_{2n+1}&=\inf{\{t\in[t^i_{2n},+\infty[,\ X^i_t\in \overline{D_{2a}}\}}\wedge\tau_{\infty}\wedge S_{\beta}.
\end{align*}
The sequence $({\theta}^i_{n})$ is non-decreasing and goes to $\tau_{\infty}\wedge S_{\beta}$ almost surely.

\begin{figure}[htbp]
\begin{center}
\includegraphics[height=70mm]{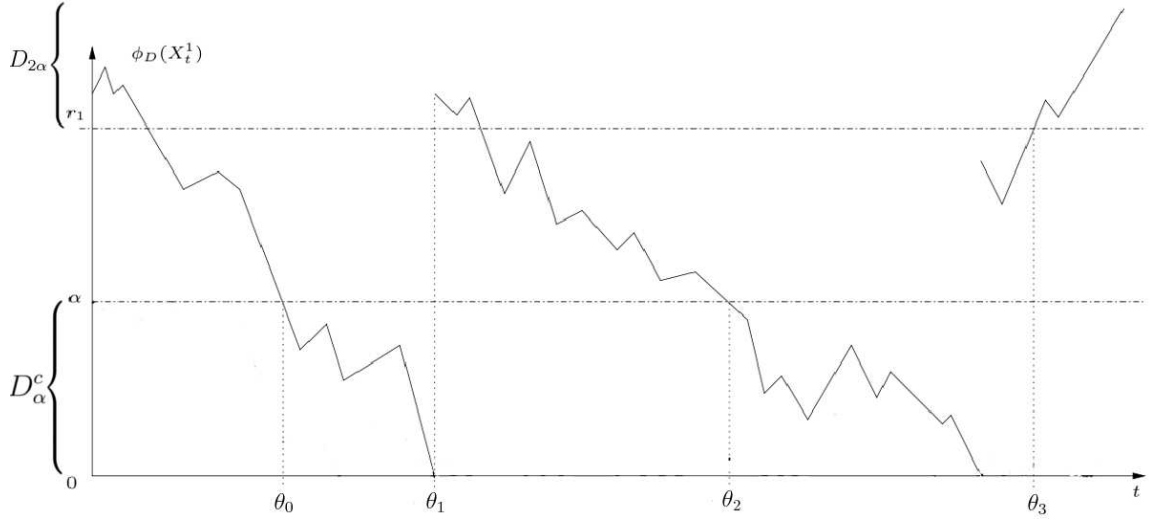}
\caption{Definition of the sequence of stopping times $({\theta}^i_n)_{n\geq0}$}
\label{fig2db}
\end{center}
\end{figure}

Let ${\gamma}^i$ be a $1$-dimensional Brownian motion independent of the process $(X^1,...,X^N)$ and of the Brownian motion $(B^1,...,B^N)$.
We set
\begin{equation*}
 W^i_t=\gamma^i_t,\ \text{for}\ t\in[0,{\theta}^i_0[,
\end{equation*}
and, for all $n\geq 0$,
\begin{align*}
 W^i_t&=W^i_{{\theta}^i_{2n}}+\ \int_{{\theta}^i_{2n}}^{t}{\nabla \phi_D(X^i_{s\minus})\cdot dB^i_s}\ \text{for}\ t\in[{\theta}^i_{2n},{\theta}^i_{2n+1}[,\\
 W^i_t&=W^i_{{\theta}^i_{2n+1}}+(\gamma^i_t-\gamma^i_{{\theta}^i_{2n+1}})\ \text{for}\ t\in[{\theta}^i_{2n+1},{\theta}^i_{2n+2}[,
\end{align*}
where $\int_{{\theta}^i_{2n}}^{t}{\nabla \phi_D(X^i_{s\minus})\cdot dB^i_s}$ has the law of a Brownian motion between times ${\theta}^i_{2n}$ and ${\theta}^i_{2n+1}$, thanks to \eqref{s12_e4}. The process $(W^1,...,W^N)$ is yet defined for all $t\in[0,\tau_{\infty}\wedge S_{\beta}[$. We set
\begin{equation*}
W^i_t=W^i_{\tau_{\infty}\wedge S_{\beta}-}+(\gamma^i_t-\gamma^i_{\tau_{\infty}\wedge S_{\beta}})\ \text{for}\ t\in [\tau_{\infty}\wedge S_{\beta},+\infty[
\end{equation*}
It is immediate that $(W^1,...,W^N)$ is a $N$-dimensional Brownian motion.

Fix $i\in\{1,...,N\}$. Thanks to Hypothesis \ref{hyp1}, there exists $Q^{(N)}_i\geq 0$ such that
\begin{equation*}
 \inf_{x\in B_{\beta}\setminus D_{2 a}}{{\cal L}_i^{(N)}\phi_D}(x)\geq -Q^{(N)}_i.
\end{equation*}
Let us prove that the reflected diffusion process $Y^{\beta,i}$ defined by \eqref{s12_e1} fulfills inequality \eqref{s12_e2} for all $t\in[0,\tau_{\infty}\wedge S_{\beta}[$.

We set $\zeta=\inf\left\{0\leq t<\tau_{\infty}\wedge S_{\beta}, Y^{\beta,i}_t > \phi_D(X^i_t)\right\}$ and we work conditionally to $\zeta<\tau_{\infty}\wedge S_{\beta}$.
By right continuity of the two processes,
\begin{equation*}
0<\phi_D(X^i_{\zeta})\leq Y^{\beta,i}_{\zeta} \leq a \ \text{a.s.} 
\end{equation*}
One can find a stopping time $\zeta'\in]\zeta,\tau_{\infty}\wedge S_{\beta}[$, such that $X^i$ doesn't jump between $\zeta$ and $\zeta'$ and such that $Y^{\beta,i}_t>0$ and $X^i_t\in B_{\beta}\setminus D_{2a}$ for all $t\in[\zeta,\zeta']$ almost surely.

Thanks to the regularity of $\phi_D$ on $B_{\beta}\setminus D_{2a}$, we can apply It\^o's formula to $(\phi_D(X^i_t))_{t\in[\zeta,\zeta']}$, and we get, for all stopping time $t\in[\zeta,\zeta']$,
\begin{align*}
 \phi_D(X^i_{t})&=\phi_D(X^i_{\zeta})+\int_{\zeta}^{t}{\nabla \phi_D(X^i_{s})\cdot dB^i_s}+\int_{\zeta}^{t}{{\cal L}_i^{(N)}\phi_D(X^i_{s})ds}.
 \end{align*}
But $\zeta$ and $\zeta'$ lie between an entry time of $X^i$ to $B_{\beta}\setminus D_a$ and the following entry time to $\overline{D_{2a}}$. It yields that there exists $n\geq 0$ such that $[\zeta,\zeta']\subset[{\theta}^i_{2n},{\theta}^i_{2n+1}[$. We deduce that
\begin{equation*}
\phi_D(X^i_t)-Y^{\beta,i}_t=\phi_D(X^i_{\zeta})-Y^{\beta,i}_{\zeta}+\int_{\zeta}^t{({\cal L}_i^{(N)}\phi_D(X^i_s)+Q^{(N)}_i)ds}-L^{i,0}_t+L^{i,0}_{\zeta}+L^{i,a}_t-L^{i,a}_{\zeta},
\end{equation*}
where ${\cal L}_i^{(N)}\phi_D(X^i_s)+Q^{(N)}_i\geq0$, $(L^{i,a}_s)_{s\geq 0}$ is increasing and $L^{i,0}_t=L^{i,0}_{\zeta}$, since $Y^{\beta,i}$ doesn't hit $0$ between times $\zeta$ and $t$. It follows that, for all $t\in[\zeta,\zeta']$,
\begin{align*}
\phi_D(X^i_t)-Y^{\beta,i}_t&\geq \phi_D(X^i_{\zeta})-Y^{\beta,i}_{\zeta}\\
                           &\geq \phi_D(X^i_{\zeta-})-Y^{\beta,i}_{\zeta-} \geq 0.
\end{align*}
where the second inequality comes from the positivity of the jumps of $\phi_D(X^i)$ and from the left continuity of $Y^{\beta,i}$, while the third inequality is due to the definition of $\zeta$.
Then $\phi_D(X^i)-Y^{\beta,i}$ stays non-negative between times $\zeta$ and $\zeta'$, what contradicts the definition of $\zeta$.
Finally, $\zeta=\tau_\infty\wedge S_{\beta}$ almost surely, which means that the coupling inequality \eqref{s12_e2} remains true for all $t\in[0,\tau_{\infty}\wedge S_{\beta}[$.
\end{proof}

\subsection{Proof of Theorem \ref{thA_1}}
\label{partA_2}
\begin{proof}[Proof that $(X^1,...,X^N)$ is well defined under Hypotheses \ref{hypD} and \ref{hyp1}]Let $N\geq2$ be the size of the interacting particle system and fix arbitrarily its starting point $x\in D^N$. Thanks to the non explosiveness of each diffusion process $\mathds{P}^i$, the interacting particle process can't escape to infinity in finite time after a finite number of jumps. It yields that $\tau_{\infty}\leq S_{\infty}$ almost surely.

 Fix $\beta>0$ such that $x\in B_{\beta}$ and define the event $C_{\beta}=\{\tau_{\infty}<S_{\beta}\}$. Assume that $C_{\beta}$ occurs with positive probability.
Conditionally to $C_{\beta}$, the total number of jumps is equal to $+\infty$ before the finite time $\tau_{\infty}$. There is a finite number of particles, then at least one particle makes an infinite number of jumps before $\tau_{\infty}$. We denote it by $i_0$ (which is a random index).

For each jumping time $\tau_n$, we denote by $\sigma_n^{i_0}$ the next jumping time of $i_0$, with $\tau_n<\sigma_n^{i_0}<\tau_{\infty}$. Conditionally to $C_{\beta}$, we get $\sigma_n^{i_0}-\tau_n\rightarrow 0$ when $n\rightarrow\infty$. For all $C^2$ function $f$ with compact support in $]0,2a[$, the process $f(\phi_D(X^{i_0}))$ is a continuous diffusion process with bounded coefficients between $\tau_n$ and $\sigma_n^{i_0}\minus$, then 
\begin{equation*}
\sup_{t\in[\tau_n,\sigma_n^{i_0}[}|f(\phi_D(X^{i_0}_{t}))|=\sup_{t\in[\tau_n,\sigma_n^{i_0}[}|f(\phi_D(X^{i_0}_{t}))-f(\phi_D(X^{i_0}_{\sigma_n^{i_0}\minus}))|\xrightarrow[n\rightarrow\infty]{} 0,\ a.s. 
\end{equation*}
Since the process $\phi_D(X^{i_0})$ is continuous between $\tau_n$ and $\sigma_n^{i_0}-$, we conclude that $\phi_D(X^{i_0}_{\tau_n})$ doesn't lie above the support of $f$, for $n$ big enough almost surely. But the support of $f$ can be chosen arbitrarily close to $0$, it yields that $\phi_D(X^{i_0}_{\tau_n})$ goes to $0$ almost surely conditionally to $C_{\beta}$.

Let us denote by $(\tau^{i_0}_n)_n$ the sequence of jumping times of the particle $i_0$. We denote by $A_n$ the event
\begin{equation*}
 \label{s12_e7}
 A_n=\left\lbrace \exists i\neq i_0\,|\, \phi_D(X^{i}_{\tau_n^{i_0}})\leq f^{(N)}(\phi_D(X^{i_0}_{\tau_n^{i_0}})) \right\rbrace,
\end{equation*}
where $f^{(N)}$ is the function of Hypothesis \ref{hyp1} .
We have, for all $1\leq k\leq l$,
\begin{align*}
 P\left( \bigcap_{n=k}^{l+1}{A_n^c} \right)&=E\left(E\left( \prod_{n=k}^{l+1}\mathds{1}_{A_n^c}\,|\,(X^1_t,...X^N_t)_{0\leq t<\tau_{l+1}^{i_0}} \right) \right)\\
                                           &=E\left(\prod_{n=k}^{l}\mathds{1}_{A_n^c}E\left(\mathds{1}_{A_{l+1}^c}\,|\,(X^1_t,...X^N_t)_{0\leq t<\tau_{l+1}^{i_0}} \right) \right),
\end{align*}
where, by definition of the jump mechanism of the interacting particle system,
\begin{align*}
 E\left(\mathds{1}_{A_{l+1}^c}\,|\,(X^1_t,...X^N_t)_{0\leq t<\tau_{l+1}^{i_0}} \right)&={\cal J}^{(N)}(X^1_{\tau_{l+1}^{i_0}},...,X^N_{\tau_{l+1}^{i_0}})\left( {A_{l+1}^c} \right)\\
                                                                                      &\leq 1-p_0^{(N)},
\end{align*}
by Hypothesis \ref{hyp1}. By induction on $l$, we get
\begin{equation*}
 P\left( \bigcap_{n=k}^{l}{A_n^c} \right)\leq (1-p_0^{(N)})^{l-k},\ \forall 1\leq k \leq l.
\end{equation*}
Since $p_0^{(N)}>0$, it yields that
\begin{equation*}
 P\left( \bigcup_{k\geq 1}\bigcap_{n=k}^{\infty}{A_n^c} \right)=0.
\end{equation*}
It means that, for infinitely many jumps $\tau_n$ almost surely, one can find a particle $j$ such that $f^{(N)}(\phi_D(X^j_{\tau_n}))\leq \phi_D(X^{i_0}_{\tau_n})$. Because there is only a finite number of other particles, one can find a particle, say $j_0$ (which is a random variable), such that
\begin{equation*}
 f^{(N)}(\phi_D(X^{j_0}_{\tau_n}))\leq \phi_D(X^{i_0}_{\tau_n}),\ \text{for infinitely many}\ n\geq 1.
\end{equation*}
In particular, $\lim_{n\rightarrow\infty}{\left(\phi_D(X^{i_0}_{\tau_n}),f^{(N)}(\phi_D(X^{j_0}_{\tau_n}))\right)}=(0,0)$ almost surely. But $(f^{(N)})^{-1}$ is well defined and continuous near $0$, then
\begin{equation*}
 \lim_{n\rightarrow\infty}{\left(\phi_D(X^{i_0}_{\tau_n}),\phi_D(X^{j_0}_{\tau_n})\right)}=(0,0)\ a.s.
\end{equation*}

Using the coupling inequality of Proposition \ref{le1}, we deduce that
\begin{equation*}
 \label{eq60}
 C_{\beta}\subset\left\{\lim_{t\rightarrow\tau_{\infty}}{(Y^{\beta,i_0}_t,Y^{\beta,j_0}_t)}=(0,0)\right\}.
\end{equation*}
Then, conditionally to $C_{\beta}$, $Y^{\beta,i_0}$ and $Y^{\beta,j_0}$ are independent reflected diffusion processes with bounded drift, which hit $0$ at the same time. This occurs for two independent reflected Brownian motions with probability $0$, and then for $Y^{\beta,i_0}$ and $Y^{\beta,j_0}$ too, by the Girsanov's Theorem. That implies $P_x(C_{\beta})=0$.

We have proved that $\tau_{\infty}\geq S_{\beta}$ almost surely for all $\beta>0$, which leads to $\tau_{\infty}\geq S_{\infty}$ almost surely. Finally, we get $\tau_{\infty}=S_{\infty}$ almost surely.

If the first part of Hypothesis \ref{hypE} is fulfilled, one can defined the coupled reflected diffusion $Y^{\infty,i}$, which fulfills inequality \eqref{s12_e2} with $a=\alpha$ and for all $t\in[0,\tau_{\infty}\wedge S_{\infty}[=[0,\tau_{\infty}[$. Then the same proof leads to
\begin{equation*}
 \{\tau_{\infty}<+\infty\}\subset\left\{\lim_{t\rightarrow\tau_{\infty}}{(Y^{\infty,i_0}_t,Y^{\infty,j_0}_t)}=(0,0)\right\}.
\end{equation*}
Finally, we deduce that $\tau_{\infty}=\infty$ almost surely.
\end{proof}

\begin{remarque}\upshape
One could wonder if the previous coupling argument can be generalized, replacing \eqref{s11_e1b} by uniformly elliptic diffusion processes. In fact, such arguments lead to the definition of $Y^i$ as the reflected diffusion $Y^i_t=\int_0^t\phi(X^i_s)dW^i_s-Q_i t+L^{0}_t-L^{\alpha}_t$, where $\phi$ is a regular function. In our case of a drifted Brownian motion, $\phi$ is equal to $1$ and $Y^i$ is a reflected drifted Brownian motion independent of the others particles. But in the general case, the $Y^i$ are general orthogonal semi-martingales. It yields that the generalization of the previous proof reduces to the following hard problem (see \cite[Question 2, page 217]{Revuz1999} and references therein): "Which are the two-dimensional continuous semi-martingales for which the one point sets are polar ?". Since this question has no general answer, it seems that the previous proof doesn't generalize immediately to general uniformly elliptic diffusion processes.

We emphasize the fact that the proof of the exponential ergodicity can be generalized (as soon as $\tau_{\infty}=S_{\infty}=+\infty$ is proved), using the fact that $(Y^1_t,...,Y^N_t)_{t\geq 0}$ is a time changed Brownian motion with drift and reflection (see \cite[Theorem 1.9 (Knight)]{Revuz1999}). This time change argument has been developed in \cite{Grigorescu2009a}, with a different coupling construction. This change of time can also be used in order to generalize Theorem 2.3 below, as soon as the exponential ergodicity is proved.
\end{remarque}

\begin{proof}[Proof of the exponential ergodicity]
It is sufficient to prove that there exists $n\geq 1$, $\epsilon >0$ and a non-trivial probability $\vartheta$ on $D^N$ such that
\begin{equation}
\label{s11_e12}
P_x((X^1_{n t^{(N)}_0},...,X^N_{n t^{(N)}_0})\in A)\geq\epsilon\vartheta(A),\ \forall x\in K_0,\ A\in{\cal B}(D^N),
\end{equation}
with $K_0=\left(K^{(N)}_0\right)^{N}$, where $t^{(N)}_0$ and $K^{(N)}_0$ are defined in Hypothesis \ref{hypE},
and such that
\begin{eqnarray}
\label{s11_e13}
\sup_{x\in K_0}{E_x(\kappa^{\tau'})}<\infty,
\end{eqnarray}
where $\kappa$ is a positive constant and $\tau'=\min\{n\geq 1, (X^1_{n t^{(N)}_0},...,X^N_{n t^{(N)}_0})_{n\in\mathds{N}}\in K_0\}$ is the return time to $K_0$ of the Markov chain $(X^1_{n t^{(N)}_0},...,X^N_{n t^{(N)}_0})_{n\in\mathds{N}}$. Indeed, Down, Meyn and Tweedie proved in \cite[Theorem 2.1 p.1673]{Down1995} that if the Markov chain $(X^1_{n t^{(N)}_0},...,X^N_{n t^{(N)}_0})_{n\in\mathds{N}}$ is aperiodic (which is obvious in our case) and fulfills \eqref{s11_e12} and \eqref{s11_e13}, then it is geometrically ergodic. But, thanks to \cite[Theorem 5.3 p.1681]{Down1995}, the geometric ergodicity of this Markov chain is a sufficient condition for $(X^1,...,X^N)$ to be exponentially ergodic.

We assume without loss of generality that $K_0^{(N)}\subset D_{\alpha/2}$ (where $\alpha$ is defined in Hypothesis \ref{hypE}). Let us set
\begin{equation*}
\vartheta(A)=\frac{\prod_{i=1}^N \inf_{x\in D_{\alpha/2}}\mathds{P}^i(X^{(i)}_{t^{(N)}_0}\in A\cap K^{(N)}_0)}{\prod_{i=1}^N \inf_{x\in D_{\alpha/2}}\mathds{P}^i(X^{(i)}_{t^{(N)}_0}\in K^{(N)}_0)}.
\end{equation*}
Thanks to Hypothesis \ref{hypE}, $\vartheta$ is a non-trivial probability measure. Moreover, \eqref{s11_e12} is clearly fulfilled with $n=1$ and $\epsilon=\prod_{i=1}^N \inf_{x\in D_{\alpha}}\mathds{P}^i(X^{(i)}_{t^{(N)}_0}\in K^{(N)}_0)$.

Let us prove that $\exists \kappa>0$ such that \eqref{s11_e13} holds. One can define the $N$-dimensional diffusion $(Y^{\infty,1},...,Y^{\infty,N})$ reflected on $\{0,\alpha\}$ and coupled with $(X^1,...,X^N)$, so that inequality \eqref{s12_e2} is fulfilled for all $t\in[0,+\infty[$ and $a=\alpha$.
For all $x_0\in D^N$, we have by the Markov property
\begin{align*}
\mathds{P}_{x_0}((X^1_{2t^{(N)}_0},...,X^N_{2t^{(N)}_0})\in K_0^N)&\geq \mathds{P}_{x_0}(X^i_{t^{(N)}_0}\in D_{\alpha/2},\forall i)\inf_{x\in D_{\alpha/2}^N}\mathds{P}_x(X^i_{t^{(N)}_0}\in K^{(N)}_0,\forall i)\\
                                   &\geq \mathds{P}_{x_0}(X^i_{t^{(N)}_0}\in D_{\alpha/2},\forall i)\prod_{i=1}^N\inf_{x\in D_{\alpha/2}}\mathds{P}_x(X^i_{t^{(N)}_0}\in K^{(N)}_0)\\
                                   &\geq \mathds{P}_{x_0}(X^i_{t^{(N)}_0}\in D_{\alpha/2},\forall i) p_1^{(N)},
\end{align*}
where $p_1^{(N)}>0$ is defined in Hypothesis \ref{hypE}.
It yields that
\begin{align*}
\mathds{P}_{x_0}((X^1_{2t^{(N)}_0},...,X^N_{2t^{(N)}_0})\in K_0^N)&\geq p_1^{(N)}\mathds{P}_{x_0}(\phi_D(X^i_{t^{(N)}_0})>\alpha/2,\forall i)\\
                                   &\geq p_1^{(N)}\prod_{i=1}^N\mathds{P}_{Y^{\infty,i}_0}(Y^{\infty,i}_{t^{(N)}_0}>\alpha/2),
\end{align*}
thanks to Proposition \ref{le1}. A comparison argument shows that $\mathds{P}_{Y^{\infty,i}_0}(Y^{\infty,i}_{t^{(N)}_0}>\alpha/2)\geq \mathds{P}_{0}(Y^{\infty,i}>\alpha/2)$. Then
\begin{equation*}
\inf_{x_0\in D}\mathds{P}_{x_0}((X^1_{2t^{(N)}_0},...,X^N_{2t^{(N)}_0})\in K_0^N)\geq p_1^{(N)}\prod_{i=1}^N\mathds{P}_{0}(Y^{\infty,i}_{t^{(N)}_0}>\alpha/2)>0,
\end{equation*}
thanks to the strict positivity of the density  of the law of $Y^{\infty,i}_{t^{(N)}_0}$, for all $i\in\{1,...,N\}$.
Using the Markov property, we get, $\forall n\geq 1$,
\begin{align*}
P(\tau'\geq 2nt^{(N)}_0)&\geq (1-\inf_{x_0\in D}\mathds{P}_{x_0}((X^1_{2t^{(N)}_0},...,X^N_{2t^{(N)}_0})\in K_0^N))P(\tau'\geq 2(n-1)t^{(N)}_0)\\
                    &\geq (1-\inf_{x_0\in D}\mathds{P}_{x_0}((X^1_{2t^{(N)}_0},...,X^N_{2t^{(N)}_0})\in K_0^N))^{n},
\end{align*}
where $0<\inf_{x_0\in D}\mathds{P}_{x_0}((X^1_{2t^{(N)}_0},...,X^N_{2t^{(N)}_0})\in K_0^N)\leq 1$. It yields that there exists $\kappa>0$ such that \eqref{s11_e13} is fulfilled.
\end{proof}

\subsection{Uniform tightness of the empirical stationary distributions}
\label{partA_3}
In this part, the open set $D$ is supposed to be bounded. Assume that a jump measure ${\cal J}^{(N)}$ and a family of drifts $(q^{(N)}_i)_{i=1,...,N}$ are given for each $N\geq 2$.
\begin{hypothese}
\label{hypQ}
Hypotheses \ref{hypD} and \ref{hyp1} are fulfilled for each $N\geq 2$ and Hypothesis \ref{hypE} is fulfilled with the same $\alpha$ for each $N\geq 2$. Moreover, there exists $r>1$ such that
\begin{equation*}
\sup_{N\geq 2}\frac{1}{N}\sum_{i=1}^{N}r^{(Q^{(N)}_i)^2}<+\infty,
\end{equation*}
where $Q^{(N)}_i=-\inf_{x\in D\setminus D_{\alpha}}{\cal L}_i^{(N)} \phi_D(x)$.
\end{hypothese}

For all $N\geq 2$, we denote by $m^N\in{\cal M}_1(D^N)$ the initial distribution and by $\mu^N(t,dx)$ the empirical distribution of the $N$-particles process defined by the jump measure ${\cal J}^{(N)}$ and the family $(q_i^{(N)})_{i\in\{1,...,N\}}$. Its stationary distribution is denoted by $M^N$ and its empirical stationary distribution is denoted by ${\cal X}^N$:
\begin{equation*}
{\cal X}^N=\frac{1}{N}\sum_{i=1}^N{\delta_{x_i}} 
\end{equation*}
where $(x^1,...,x^N)$ is a random vector in $D^N$ distributed following $M^N$.

\begin{theoreme}
\label{thB}
Assume that Hypothesis \ref{hypQ} is fulfilled. For all sequence of measures $m^N\in{\cal M}_1(D^N)$ and all $t>0$, the family of random measures $\left(\mu^{N}(t,dx)\right)_{N\geq 2}$ is uniformly tight. In particular, the family of empirical stationary distributions $\left({\cal X}^N\right)_{N\geq 2}$ is uniformly tight.
\end{theoreme}

\begin{proof}
 Let us consider the process $(X^1,...,X^N)$ starting with a distribution $m^N$ and its coupled process $(Y^{\infty,1},...,Y^{\infty,N})$.
  For all $t\in[0,\tau_{\infty}[$, we denote by  $\mu'^N(t,dx)$ the empirical measure of $(Y^{\infty,1}_t,...,Y^{\infty,N}_t)$.
 By the coupling inequality \eqref{s12_e2}, we get
\begin{equation*}
 \mu^N(t,D_r^c)\leq\mu'^N(t,[0,r]),\ \forall r\in[0,\alpha].
\end{equation*}
 Using the Markov property, we deduce that, for all $s<t$,
\begin{align*}
E_{X^1_{s},...,X^N_{s}}\left(\mu^N(t-s,D_r^c)\right)&\leq E_{Y^{\infty,1}_{s},...,Y^{\infty,N}_{s}}\left(\mu'^N(t-s,D_r^c)\right) \text{a.s.}
\end{align*}
Then, by a comparison argument,
\begin{align}
E_{X^1_{s},...,X^N_{s}}\left(\mu^N(t-s,D_r^c)\right)&\leq E_{0,...,0}\left(\mu'^N(t-s,D_r^c)\right) \ \text{a.s.}\nonumber\\
                      &\leq \frac{1}{N}\sum_{i=1}^N P_0(Y^{\infty,i}_{t-s}\leq r)\ \text{a.s.} \label{ceq6}
\end{align}
Thanks to the Girsanov's Theorem, we have
\begin{equation*}
P_0(Y^{\infty,i}_{t-s}\leq r)= E_{0}\left(\delta_{w^i_{t-s}+L_{t-s}^{i,\alpha}-L_{t-s}^{i,0}}([0,r])e^{Q^{(N)}_i w^i_{t-s} -(Q^{(N)}_i)^2 ({t-s})}\right)e^{\frac{3}{2}(Q^{(N)}_i)^2 ({t-s})},
\end{equation*}
where $(w^1,...,w^N)$ is a $N$-dimensional Brownian motion.
By the Cauchy Schwartz inequality, we get
\begin{align*}
P_0(Y^{\infty,i}_{t-s}\leq r)&\leq \sqrt{E_{0}\left(\left(\delta_{w^i_{t-s}+L_{t-s}^{i,\alpha}-L_{t-s}^{i,0}}([0,r])\right)^2\right)E_{0}\left(\left(e^{Q^{(N)}_i w^i_{t-s} - (Q^{(N)}_i)^2 ({t-s})}\right)^2\right)},\\
                   &\leq \sqrt{E_{0}\left(\delta_{w^i_{t-s}+L_{t-s}^{i,\alpha}-L_{t-s}^{i,0}}([0,r])\right)}
\end{align*}
where the second inequality occurs, since $0\leq \delta_{w^i_{t-s}+L_{t-s}^{i,\alpha}-L_{t-s}^{i,0}}([0,r]) \leq 1$ almost surely and the process $e^{2 Q^{N}_i w^i_t - 2 (Q^{(N)}_i)^2 t}$ is the Dol\'eans exponential of $2Q^{(N)}_i w^i_t$, whose expectation is $1$. Taking the expectation in \eqref{ceq6}, it yields that
\begin{align*}
E_{m^N}\left(\mu^N(t,D_r^c)\right)&\leq \sqrt{P_0\left(\delta_{w^i_{t-s}+L_{t-s}^{i,\alpha}-L_{t-s}^{i,0}}([0,r])\right)}\frac{1}{N}\sum_{i=1}^N e^{\frac{3}{2}(Q_i^{(N)})^2(t-s)}, \ \forall 0<s<t.
\end{align*}
Thanks to Hypothesis \ref{hypQ}, there exists $s_0\in]0,t[$ such that $\frac{1}{N}\sum_{i=1}^N e^{\frac{3}{2}(Q_i^{(N)})^2(t-s_0)}$ is uniformly bounded in $N\geq 2$. But $P_0\left(\delta_{w^i_{t-s_0}+L_{t-s_0}^{i,\alpha}-L_{t-s_0}^{i,0}}([0,r])\right)$ goes to $0$ when $r\rightarrow 0$, so that the family of random measures $(\mu^N(t,dx))_{N\geq 2}$ is uniformly tight. 

If we set $m^N$ equal to the stationary distribution $M^N$, then
we get by stationarity that ${\cal X}^N$ is distributed as $\mu^N(t,.)$, for all $N\geq 2$ and $t>0$. Finally, the family of empirical stationary distributions $({\cal X}^N)_{N\geq 2}$ is uniformly tight.
\end{proof}

\section{Yaglom limit's approximation}
\label{partApproximation}
We consider now the particular case ${\cal J}^{(N)}(x_1,...,x_N)=\frac{1}{N-1}\sum_{k=1,k\neq i}^{N}{\delta_{x_k}}$: at each jump time, the particle which hits the boundary jumps to the position of a particle chosen uniformly between the $N-1$ remaining ones. We assume moreover that $q^{(N)}_i=q$ doesn't depend on $i,N$. In this framework, we are able to identify the limiting distribution of the empirical stationary distribution sequence, when the number of particles tends to infinity. This leads us to an approximation method of the Yaglom limits \eqref{intyag1}, including cases where the drift of the diffusion process isn't bounded and where the boundary is neither regular nor bounded.

Let $U_{\infty}$ be an open domain of $\mathds{R}^d$, with $d\geq 1$.
We denote by $\mathds{P}^{\infty}$ the law of the diffusion process defined on $U_{\infty}$ by
\begin{equation}
\label{s21_e1}
 dX^{\infty}_t=dB_t-\nabla V(X^{\infty}_t)dt,\ X^{\infty}_0=x\in U_{\infty}
\end{equation}
and absorbed at the boundary $\partial U_{\infty}$. Here $B$ is a $d$-dimensional Brownian motion and $V\in C^2(U_{\infty},\mathds{R})$. 
We assume that Hypothesis \ref{H3} below is fulfilled, so that the Yaglom limit 
\begin{equation}
\label{s21_e2}
 \nu_{\infty}=\lim_{t\rightarrow+\infty} \mathds{P}^{\infty}_x\left(X^{\infty}_t\in .| t\leq \tau_{\partial} \right), \forall x\in U_{\infty}
\end{equation}
exists and doesn't depend on $x$, as proved by Cattiaux and M\'el\'eard in \cite[Theorem B.2]{Cattiaux2008}.
We emphasize the fact that this hypothesis allows the drift $\nabla V$ of the diffusion process \eqref{s21_e1} to be unbounded and the boundary $\partial U_{\infty}$ to  be neither of class $C^2$ nor bounded. In particular, the results of the previous section aren't available in all generality for diffusion processes with law $\mathds{P}^{\infty}$.
\begin{hypothese}
\label{H3}
We assume that
 \begin{enumerate}
 \item $\mathds{P}^{\infty}_x(\tau_{\partial}<+\infty)=1$,
 \item $\exists C>0$ such that $G(x)=|\nabla V|^2(x)-\Delta V(x)\geq -C > -\infty$, $\forall x\in U_{\infty}$,
 \item $\overline{G}(R)\rightarrow+\infty$ as $R\rightarrow\infty$, where
\begin{equation*}
 \overline{G}(R)=\inf\left\{ G(x); |x|\geq R\ \text{and}\  x\in U_{\infty} \right\},
\end{equation*}
 \item There exists an increasing sequence $(U_{m})_{m\geq0}$ of bounded open subsets of $U_{\infty}$, such that the boundary of $U_m$ is of class $C^2$ for all $m\geq 0$, and such that $\bigcup_{m\geq 0}\overline{U_m}=U_{\infty}$.
 \item There exists $R_0>0$ such that
\begin{multline*}
 \int_{U_{\infty}\cap \{d(x,\partial U_{\infty})>R_0\}}{e^{-2V(x)}dx}<\infty\ \text{and}\\
        \ \int_{U_{\infty}\cap \{d(x,\partial U_{\infty})\leq R_0\}}{\left( \int_{U_{\infty}}{p_1^{U_{\infty}}(x,y) dy} \right)e^{-V(x)}dx}<\infty.
\end{multline*}
Here $p_1^{U_{\infty}}$ is the transition density of the diffusion process \eqref{s21_e1} with respect to the Lebesgue measure.
\end{enumerate}
\end{hypothese}
 According to \cite{Cattiaux2008}, the second point implies that the semi-group induced by $\mathds{P}^{\infty}$ is ultra-contractive. The assumptions 1-4 imply that the generator associated with $\mathds{P}^{\infty}$ has a purely discrete spectrum and that its minimal eigenvalue $-\lambda_{\infty}$ is simple and negative. The last assumption ensures that the eigenfunction associated with $-\lambda_{\infty}$ is integrable with respect to $e^{-2V(x)}dx$. Finally, Hypothesis \ref{H3} is sufficient for the existence of the Yaglom limit \eqref{s21_e2}.

\begin{remarque}\upshape
 For example, it is proved in \cite{Cattiaux2008} that Hypothesis \ref{H3} is fulfilled by the Lotka-Volterra system studied numerically in Subsection \ref{partApproximation_3}. Up to a change of variable, this system is defined by the diffusion process with values in $U_{\infty}=\mathds{R}_+^2$, which satisfies
\begin{equation}
\label{eqr1}
 \begin{split}
  dY^1_t=dB^1_t+\left(\frac{r_1 Y^1_t}{2}-\frac{c_{11}\gamma_1\left( Y^1_t \right)^3}{8} -\frac{c_{12}\gamma_2 Y_t^1\left( Y^2_t \right)^2}{8} -\frac{1}{2 Y_t^1}  \right)dt\\
  dY^2_t=dB^2_t+\left(\frac{r_2 Y^2_t}{2}-\frac{c_{22}\gamma_2\left( Y^2_t \right)^3}{8} -\frac{c_{21}\gamma_1 Y_t^2\left( Y^1_t \right)^2}{8} -\frac{1}{2 Y_t^2}  \right)dt
 \end{split}
\end{equation}
and is absorbed at $\partial U_{\infty}$. Here $B^1,B^2$ are two independent one-dimensional Brownian motions and the parameters of the diffusion process fulfill condition \eqref{eql3}.
\end{remarque}

In order to define the interacting particle process of the previous section, we work with diffusion processes defined on $U_m$, $m\geq 0$.
More precisely, for all $m\geq 0$, we denote by $\mathds{P}^{m}$ the law of the diffusion process defined on $U_m$ by
\begin{equation}
\label{ceq3}
 dX^{U_m}_t=dB_t-q_m(X^{U_m}_t)dt,\ X_0^{U_m}=x\in U_m
\end{equation}
and absorbed at the boundary $\partial U_m$. Here $B$ is a $d$-dimensional Brownian motion and $q_m:\overline{U_m}\mapsto \mathds{R}$ is a continuous function.  We denote by ${\cal L}^{m}$ the infinitesimal generator of the diffusion process with law $\mathds{P}^{m}$.
For all $m\geq 0$, the diffusion process with law $\mathds{P}^{m}$ clearly fulfills the conditions of Section \ref{partA}. For all $N\geq 2$, let $(X^{m,1},...,X^{m,N})$ be the interacting particle process defined by the law $\mathds{P}^{m}$ between the jumps and by the jump measure ${\cal J}^{(m,N)}(x_1,...,x_N)=\frac{1}{N-1}\sum_{k=1,k\neq i}^{N}{\delta_{x_k}}$. By Theorem \ref{thA_1}, this process is well defined and exponentially ergodic.

For all $m\geq 0$ and all $N\geq 2$, we denote by $\mu^{m,N}(t,dx)$ the empirical distribution of $(X^{m,1}_t,...,X^{m,N}_t)$, by $M^{m,N}$ the stationary distribution of $(X^{m,1},...,X^{m,N})$ and by ${\cal X}^{m,N}$ the associated empirical stationary distribution.

We are now able to state the main result of this section.
\begin{theoreme}
\label{th3_1}
 Assume that Hypothesis \ref{H3} is satisfied and that $q_m=\nabla V\mathds{1}_{\overline{U_m}}$ for all $m\geq 0$. Then
 \begin{equation*}
  \lim_{m\rightarrow \infty} \lim_{N\rightarrow\infty} {\cal X}^{m,N}=\nu_{{\infty}},
 \end{equation*}
in the weak topology of random measures, which means that, for all bounded continuous function $f:U_{\infty}\mapsto\mathds{R_+}$,
\begin{equation*}
\lim_{m\rightarrow \infty} \lim_{N\rightarrow\infty} E({\cal X}^{m,N}(f))=\nu_{{\infty}}(f).
\end{equation*}
\end{theoreme}

In Section \ref{partApproximation_1}, we fix $m\geq 0$ and we prove that the sequence $({\cal X}^{m,N})_{N\geq 2}$ converges to a deterministic probability $\nu_{m}$ when $N$ goes to infinity. In particular, we prove that $\nu_{m}$ is the Yaglom limit associated with $\mathds{P}^{m}$, which exists by \cite{Cattiaux2008}.
In Section \ref{partApproximation_2}, we conclude the proof, proceeding by a compactness/uniqueness argument: we prove that $(\nu_{m})_{m\geq 0}$ is a uniformly tight family and we show that each limiting probability of the family $(\nu_{m})_{m\geq 0}$ is equal to the Yaglom limit $\nu_{\infty}$. 
The last Section \ref{simulation} is devoted to numerical illustrations of Theorem \ref{th3_1}.

\subsection{Convergence of $({\cal X}^{m,N})_{N\geq 2}$, when $m\geq 0$ is fixed}
\label{partApproximation_1}

\begin{proposition}
\label{proApproximation_1}
 Let $m\geq 0$ be fixed and let $q_m:\overline{U_m}\mapsto\mathds{R}$ be a continuous function. Assume that $\mu^{m,N}(0,dx)$ converges in the weak topology of random measure to a random probability measure $\mu_m$ with values in ${\cal M}_1(U_m)$, when $N\rightarrow\infty$. Then, for all $T\geq 0$, $\mu^{m,N}(T,dx)$ converges in the weak topology of random measure to $\mathds{P}^m_{\mu_m}(X_T\in .|X_T\in U_m)$ when $N$ goes to infinity.

 Moreover, if there exists $\nu_m\in{\cal M}_1(U_m)$ such that
\begin{equation}
\label{s21_e6}
 \nu_{m}=\lim_{t\rightarrow\infty}\mathds{P}^{m}_{\mu}\left( X^{U_m}_t \in . | X^{U_m}_t\in U_{m}\right),\ \forall \mu\in{\cal M}_1(U_m),
\end{equation}
then the sequence of empirical stationary distributions $({\cal X}^{m,N})_{N\geq 2}$ converges to $\nu_{m}$ in the weak topology of random measures when $N$ goes to infinity.
\end{proposition}

\begin{remarque}\upshape
In Proposition \ref{proApproximation_1}, $\nu_m$ is the Yaglom limit and the unique quasi-stationary distribution associated with $\mathds{P}^{m}$. For instance, each of the two following conditions is sufficient for the existence of such a measure:
\begin{enumerate}
\item If $q_m=\mathds{1}_{\overline{U_m}}\nabla V$, by \cite{Cattiaux2008}. This is the case of Theorem \ref{th3_1}.
\item If $q_m$ belongs to $C^{1,\alpha}(\overline{U_m})$ with $\alpha>0$, by \cite{Gong1988}.
\end{enumerate}
\end{remarque}

\begin{proof}[Proof of Proposition \ref{proApproximation_1}]
We set
\begin{eqnarray*}
\nu^{m,N}(t,dx)=\left(\frac{N-1}{N}\right)^{A^N_t}\mu^{m,N}(t,dx),
\end{eqnarray*}
where $A^N_t=\sum_{n=1}^{\infty}\mathds{1}_{\tau_n\leq t}$ denotes the number of jumps before time $t$. Intuitively,
we introduce a loss of $1/N$ of the total mass at each jump, in order to approximate the distribution of the diffusion process \eqref{ceq3} without conditioning. We will come back to the study of $\mu^{m,N}$ and the conditioned diffusion process by normalizing $\nu^{m,N}$.

 From the It\^o's formula applied to the semi-martingale $\mu^{m,N}(t,\psi)=\frac{1}{N}\sum_{i=1}^{N}\psi(X^{m,i}_t)$, where $\psi\in C^2(U_m,\mathds{R})$, we get
\begin{multline}
\label{th4eq2}
 \mu^{m,N}(t,\psi)=\mu^{m,N}(0,\psi)+\int_0^t{\mu^{m,N}(s,{\cal L}^{m}\psi)ds}+{\cal M}^{c,N}(t,\psi)+{\cal M}^{j,N}(t,\psi)\\+\frac{1}{N-1}\sum_{0\leq \tau_n\leq t}{\mu^{m,N}(\tau_n\minus,\psi)},
\end{multline}
where ${\cal M}^{c,N}(t,\psi)$ is the continuous martingale 
\begin{equation*}
\frac{1}{N}\sum_{i=1}^N{\sum_{j=1}^d{\int_0^t{\frac{\partial\psi}{\partial x_j}(X^{m,i}_{s})}dB^{i,j}_s}}
\end{equation*}
and ${\cal M}^{j,N}(t,\psi)$ is the pure jump martingale
\begin{equation*}
 \frac{1}{N}\sum_{i=1}^N{\sum_{0\leq \tau^i_n\leq t}{\left(\psi(X^{m,i}_{\tau^i_n})-\frac{N}{N-1}\mu^{m,N}(\tau^i_n\minus,\psi)\right)}}.
\end{equation*}
Applying the It\^o's formula to the semi-martingale $\nu^{m,N}(t,\psi)$, we deduce from \eqref{th4eq2} that
\begin{align*}
 \nu^{m,N}(t,\psi)=\nu^{m,N}(0,\psi)+\int_0^t{\nu^{m,N}(s,{\cal L}^{m}\psi)ds}&+\int_0^t{\left(\frac{N-1}{N}\right)^{A^N_{s}}d{\cal M}^{c,N}(s,\psi)}\\&+\sum_{0\leq\tau_n\leq t}(\nu^{m,N}(\tau_n,\psi)-\nu^{m,N}(\tau_n\minus,\psi)).
\end{align*}
Where we have
\begin{multline*}
 \nu^{m,N}(\tau_n,\psi)-\nu^{m,N}(\tau_n\minus,\psi)=\left(\frac{N-1}{N}\right)^{A^N_{\tau_n}}\left(\mu^{m,N}(\tau_n,\psi)-\mu^{m,N}(\tau_n\minus,\psi)\right)\\
+\mu^{m,N}(\tau_n\minus,\psi)\left(\left(\frac{N-1}{N}\right)^{A^N_{\tau_n}}-\left(\frac{N-1}{N}\right)^{A^N_{\tau_n\minus}}\right).
\end{multline*}
But
\begin{eqnarray*}
 \label{eqd12}
 \mu^{m,N}(\tau_n,\psi)-\mu^{m,N}(\tau_n\minus,\psi)=\frac{1}{N-1}\mu^{m,N}(\tau_n\minus,\psi)+{\cal M}^{j,N}(\tau_n,\psi)-{\cal M}^{j,N}(\tau_n\minus,\psi)
\end{eqnarray*}
and
\begin{eqnarray*}
 \label{eqd10}
 \left(\frac{N-1}{N}\right)^{A^N_{\tau_n}}-\left(\frac{N-1}{N}\right)^{A^N_{\tau_n\minus}}=-\frac{1}{N-1}\left(\frac{N-1}{N}\right)^{A^N_{\tau_n}}.
\end{eqnarray*}
Then
\begin{eqnarray*}
 \nu^{m,N}(\tau_n,\psi)-\nu^{m,N}(\tau_n\minus,\psi)&=&\left(\frac{N-1}{N}\right)^{A^N_{\tau_n}}\left({\cal M}^{j,N}(\tau_n,\psi)-{\cal M}^{j,N}(\tau_n\minus,\psi) \right).\\
    &=&\frac{N-1}{N}\left(\frac{N-1}{N}\right)^{A^N_{\tau_n\minus}}\left({\cal M}^{j,N}(\tau_n,\psi)-{\cal M}^{j,N}(\tau_n\minus,\psi) \right).
\end{eqnarray*}
That implies
\begin{multline*}
 \nu^{m,N}(t,\psi)-\nu^{m,N}(0,\psi)=\int_0^t{\nu^{m,N}(s,{\cal L}^{m}\psi)}ds+\int_0^t{\left(\frac{N-1}{N}\right)^{A^N_{s}}d{\cal M}^{c,N}(s,\psi)}\\
+\frac{N-1}{N}\sum_{0\leq\tau_n\leq t}{\left(\frac{N-1}{N}\right)^{A^N_{\tau_n\minus}}\left({\cal M}^{j,N}(\tau_n,\psi)-{\cal M}^{j,N}(\tau_n\minus,\psi)\right)}.
\end{multline*}
It yields that,
for all smooth functions $\Psi(t,x)$ vanishing at the boundary of $U_m$,
\begin{equation*}
\begin{split}
\nu^{m,N}(t,\Psi(t,.))-\nu^{m,N}(0,\Psi(0,.))=\int_0^t{\nu^{m,N}(s, \frac{\partial \Psi(s,.)}{\partial s}+{\cal L}^{m}\Psi(s,.))}ds\\
+{\cal N}^{c,N}(t,\Psi)+{\cal N}^{j,N}(t,\Psi),
\end{split}
\end{equation*}
where ${\cal N}^{c,N}(t,\Psi)$ is the continuous martingale
\begin{equation*}
 \frac{1}{N}\sum_{i=1}^N{\sum_{j=1}^d{\int_0^t{\left(\frac{N-1}{N}\right)^{A^N_{s}}\frac{\partial\Psi}{\partial x_j}(s,X^{m,i}_{s})}dB^{i,j}_s}}
\end{equation*}
and ${\cal N}^{j,N}(t,\Psi)$ is the pure jump martingale
\begin{equation*}
 \frac{1}{N}\sum_{i=1}^N{\sum_{0\leq \tau^i_n\leq t}{\left(\frac{N-1}{N}\right)^{A^N_{\tau^i_n\minus}}\left(\Psi(\tau^i_n,X^{i}_{\tau^i_n})-\frac{N}{N-1}\mu^{m,N}(\tau^i_n\minus,\Psi(\tau^i_n\minus,.))\right)}}.
\end{equation*}
Let $T>0$ be fixed. For all $\delta>0$, define $\Psi^{\delta}(t,x)=P^{m}_{T-t}P^{m}_{\delta}f(x)$, where $f\in C^{2}(\overline{U_m})$ and $(P^{m}_s)_{s\geq 0}$ is the semigroup associated with $\mathds{P}^{m}$ : $P^{m}_s f(x)=E_x(f(X^{U_m}_s))$. Then $\Psi^{\delta}$ vanishes on the boundary, is smooth, and fulfills
\begin{eqnarray*}
 \frac{\partial}{\partial s}\Psi^{\delta}(s,x)+\frac{1}{2}\Delta\Psi^{\delta}(s,x)+q_m(x)\nabla\Psi^{\delta}(s,x)=0,
\end{eqnarray*}
thanks to Kolmogorov's equation (see \cite[Proposition 1.5 p.9]{Ethier1986}).
It yields that
\begin{equation}
\label{eq20}
\nu^{m,N}(t,\Psi^{\delta}(t,.))-\nu^{m,N}(0,\Psi^{\delta}(0,.))={\cal N}^{c,N}(t,\Psi^{\delta})+{\cal N}^{j,N}(t,\Psi^{\delta}).
\end{equation}
Since $\left(\frac{N-1}{N}\right)^{A^N_{s}}\leq 1$ a.s., we get
\begin{equation}
\label{eqs1}
\begin{split}
E\left({\cal N}^{c,N}(T,\Psi^{\delta})^2\right)&\leq \frac{T}{N}\|\nabla \Psi^{\delta}\|_{\infty}^2\\
                                                      &\leq \frac{T}{N} \frac{c_{m}}{\sqrt{(T-t+\delta)\wedge 1}} \|f\|_{\infty}^2
\end{split}
\end{equation}
where $c_{m}>0$ is a positive constant. The last inequality comes from \cite[Theorem 4.5]{Priola2006} on gradient estimates in regular domains of $\mathds{R}^d$.
The jumps of the martingale ${\cal M}^{j,N}(t,\Psi^{\delta})$ are smaller than $\frac{2}{N}\|\Psi^{\delta}\|_{\infty}$, then
\begin{align*}
 E\Bigg[\sum_{0\leq \tau_n \leq T}{\left(\frac{N-1}{N}\right)^{2 A_{\tau_n\minus}}}&{\left({\cal M}^{j,N}(\tau_n,\Psi^{\delta}(\tau_n,.))-{\cal M}^{j,N}(\tau_n\minus,\Psi^{\delta}(\tau_n\minus,.))\right)^2}\Bigg]\\
&\leq \frac{4}{N^2}\|\Psi^{\delta}\|_{\infty}^2 E\left[\sum_{0\leq \tau_n \leq T}{\left(\frac{N-1}{N}\right)^{2 A_{\tau_n\minus}}}\right]\\
&\leq \frac{4}{N}\|\Psi^{\delta}\|_{\infty}^2 .
\end{align*}
Then
\begin{equation}
\label{eqs2}
 E\left({\cal N}^{j,N}(\Psi,T)^2\right)\leq \frac{4}{N}\|\Psi\|_{\infty}^2 \leq \frac{4}{N}\|f\|_{\infty}^2.
\end{equation}
Taking $t=T$ and $\delta=\frac{1}{N}$, we get from \eqref{eq20}, \eqref{eqs1} and \eqref{eqs2} that
\begin{eqnarray*}
\sqrt{E\left(\left| \nu^{m,N}(t,P^{m}_{\frac{1}{N}} f)- \nu^{m,N}(0, P^{m}_{T+\frac{1}{N}} f )\right|^2\right)}\leq \sqrt{\frac{c_m T+4}{\sqrt{N}}}\|f\|_{\infty}.
\end{eqnarray*}
Assume that $f$ vanishes at $\partial U_m$, so that $f$ belongs to the domain of ${\cal L}^m$. Then $\|P^{m}_{\frac{1}{N}}f - f\|_{\infty}\leq \frac{1}{N}\|{\cal L}^m f\|_{\infty}$ and we have
\begin{equation}
 \label{eq22}
\sqrt{ E\left(\left|\nu^{m,N}(T,f)-\nu^{m,N}(0, P^{m}_{T} f) \right|^2\right)}\leq \sqrt{\frac{c_m T+4}{\sqrt{N}}}\|f\|_{\infty}+\frac{2}{N}\|{\cal L}^m f\|_{\infty}\xrightarrow{N\rightarrow\infty} 0.
\end{equation}
By assumption, the family of random probabilities $(\nu^{m,N}(0,.))_{N\geq 2}=(\mu^{m,N}(0,.))_{N\geq 2}$ converges to $\mu_m$. We deduce from \eqref{eq22} that
\begin{equation}
\label{s21_e3}
 E\left(\nu^{m,N}(T,f)\right)\xrightarrow[N\rightarrow\infty]{}E\left({\mu_m}^{}(P^{m}_T f)\right),
\end{equation}
for all $f\in C^{2}(U_m)$ vanishing at boundary.
But the family $\left(\nu^{m,N}(T,.)\right)_{N\geq 2}$ is uniformly tight by Theorem \ref{thB} . It yields from \eqref{s21_e3} that its unique limiting distribution is $\mu_m(P^m_T .)$. In particular, 
\begin{equation*}
\left(\nu^{m,N}(T,U_m),\nu^{m,N}(T,.)\right)\xrightarrow[N\rightarrow\infty]{law}\left(\mu_m(P^{m}_T \mathds{1}_{U_m}),\mu_m(P^{m}_T .)\right). 
\end{equation*}
But $\mu_m^{}(P^{m}_T \mathds{1}_{U_m})$ never vanishes almost surely, so that
\begin{equation}
\label{ce1}
 \mu^{m,N}(T,.)=\frac{\nu^{m,N}(T,.)}{\nu^{m,N}(T,U_m)}\xrightarrow[N\rightarrow\infty]{law}\frac{\mu_m(P^{m}_T .)}{{\mu_m(P^{m}_T \mathds{1}_{U_m})}}=\mathds{P}^{m}_{\mu_m}(X^{U_m}_T\in . | X^{U_m}_T\in U_m).
\end{equation}

The family of random probabilities  $({\cal X}^{m,N})_{N\geq 0}$ is uniformly tight, by Theorem \ref{thB}. Let ${\cal X}^{m}$ be one of its limiting probabilities. By definition, there exists a strictly increasing map $\varphi:\mathds{N}\mapsto \mathds{N}$, such that ${\cal X}^{m,\varphi(N)}$ converges in distribution to ${\cal X}^{m}$ when $N\rightarrow\infty$. By stationarity, ${\cal X}^{m,\varphi(N)}$ has the same law as $\mu^{m,\varphi(N)}(T,.)$, which converges in distribution to $\mathds{P}^{m}_{\cal X}(X^{U_m}_T\in . | X^{U_m}_T\in U_m)$, thanks to \eqref{ce1}. But $\mathds{P}^{m}_{{\cal X}^{m}}(X^{U_m}_T\in . | X^{U_m}_T\in U_m)$ converges almost surely to $\nu_{m}$ when $T\rightarrow\infty$, by \eqref{s21_e6}. We deduce from this that ${\cal X}^{m}$ has the same law as $\nu_{m}$. As a consequence, the unique limiting probability of the uniformly tight family $({\cal X}^{m})_N$ is $\nu_{m}$, which allows us to conclude the proof of Proposition \ref{proApproximation_1}.
\end{proof}

\subsection{Convergence of the family $(\nu_{m})_{m\geq 0}$}
\label{partApproximation_2}

\begin{proposition}
\label{pr2}
 Assume that Hypothesis \ref{H3} is fulfilled and that $q_m=\nabla V \mathds{1}_{\overline{U_m}}$. Then the sequence $(\nu_{m})_{m\geq 0}$ converges weakly to the Yaglom limit $\nu_{\infty}$ when $m\rightarrow\infty$.
\end{proposition}

\begin{remarque}\upshape
Since $q_m=\nabla V \mathds{1}_{\overline{U_m}}$, the operator ${\cal L}^m$ is symmetric with respect to the measure $e^{-2V(x)}dx$, but this isn't directly used in the proof of Proposition \ref{pr2}. We mainly use inequalities from \cite{Cattiaux2008} that are implied by the ultra-contractivity of $\mathds{P}^{\infty}$ and the third point of Hypothesis \ref{H3}. However, it seems hard to generalize this last hypothesis and its implications to diffusions with non-gradient drifts.
\end{remarque}

\begin{proof}[Proof of Proposition \ref{pr2}]
For all $m\geq 0$ and $m=\infty$, it has been proved in \cite{Cattiaux2008} that $-{\cal L}^{m*}$ has a simple eigenvalue $\lambda_{m}>0$ with minimal real part, where ${\cal L}^{m*}$ is the adjoint operator of ${\cal L}^m$. The corresponding normalized eigenfunction $\eta_{m}$ is strictly positive on $U_{m}$, belongs to $C^2(U_{m},\mathds{R})$ and fulfills  
\begin{equation}
\label{eq23}
{\cal L}^{m*}\eta_{m}=-\lambda_{m}\eta_{m}\ \text{and}\ \int_{U_{m}}{\eta_{m}(x)^2d\sigma(x)}=1,
\end{equation}
where
\begin{equation*}
d\sigma(x)=e^{-2V(x)}dx.
\end{equation*}
The Yaglom limit $\nu_{m}$ is given by
\begin{equation*}
 d\nu_{m}=\frac{\eta_{m}\mathds{1}_{U_m} d\sigma}{\int_{U_{m}}{\eta_{m}(x) d\sigma(x)}},\ \forall m\geq 0\ \text{or}\ m=\infty.
\end{equation*}
In order to prove that $(\nu_{m})_{m\geq 0}$ converges to $\nu_{\infty}$, we show that $(\lambda_m)_{m\geq 0}$ converges to $\lambda_{\infty}$. Then we prove that $\left(\eta_{m}\mathds{1}_{U_m} d\sigma\right)_{m\geq 0}$ is uniformly tight. We conclude by proving that every limiting point $\eta d\sigma$ is a nonzero measure proportional to $\eta_{\infty} d\sigma$.

For all $m\geq 0$ or $m=\infty$, the eigenvalue $\lambda_{m}$ of $-{\cal L}^{m*}$ is given by (see for instance \cite[chapter XI, part 8]{Yosida1968})
\begin{equation*}
 \lambda_{m}=\inf_{\phi\in C^{\infty}_0\left(U_m\right),\ \left\langle\phi,\phi\right\rangle_{\sigma,m}=1}{\left\langle {\cal L}^{m*}\phi,\phi\right\rangle_{\sigma,m}}.
\end{equation*}
where $C^{\infty}_0\left(U_m\right)$ is the vector space of infinitely differentiable functions with compact support in $U_m$ and  $\left\langle f,g \right\rangle_{\sigma,m}=\int_{U_{m}} {f(u)g(u)d\sigma(u)}$. For all $\phi\in C^{\infty}_0\left(U_{\infty}\right)$, the support of $\phi$ belongs to $U_m$ for $m$ big enough, then $C^{\infty}_0\left(U_{\infty}\right)=\bigcup_{m\geq 0}C^{\infty}_0\left(U_{m}\right)$ since the reverse inclusion is clear. Moreover, if $\phi\in C^{\infty}_0\left(U_{m}\right)$, then ${\cal L}^{\infty*}\phi(x)={\cal L}^{m*}\phi(x)$ for all $x\in U_m$. Finally,
\begin{align*}
 \lambda_{\infty}&=\inf_{m\geq 0}\inf_{\phi\in C^{\infty}_0\left(U_m\right),\ \left\langle\phi,\phi\right\rangle_{\sigma,m}=1}\left\langle {\cal L}^{m*}\phi,\phi\right\rangle_{\sigma,m}\\
                 &=\lim_{m\geq 0}\searrow \lambda_{m}.
\end{align*}

Let us show that the family $(\eta_{m}\mathds{1}_{U_m}d\sigma)_{m\geq 0}$ is uniformly tight.
Fix an arbitrary positive constant $\epsilon>0$ and let us prove that one can find a compact set $K_{\epsilon}\subset U_{\infty}$ which fulfills
\begin{equation}
\label{s21_e11}
 \int_{{U_{\infty}}{\setminus K_{\epsilon}}}{\epsilon_{m}\mathds{1}_{U_m}d\sigma}\leq \epsilon,\ \forall m\geq 0.
\end{equation}
Let $R_0$ be the positive constant of the fifth part of Hypothesis \ref{H3}. For all compact set $K$, we have
\begin{equation*}
 \int_{U_{\infty}\setminus K}{\eta_{m}\mathds{1}_{U_m}d\sigma}=\int_{\{d(x,\partial U_{m}) > R_0\}\cap U_m\setminus K}{\eta_{m}d\sigma}+\int_{\{d(x,\partial U_{m}) \leq R_0\}\cap U_m\setminus K}{\eta_{m}d\sigma}.
\end{equation*}
From the proof of \cite[Proposition B.6]{Cattiaux2008}, we have on the one hand
\begin{align*}
 \int_{\{d(x,\partial U_{m}) > R_0\}\cap U_m\setminus K}{\eta_{m}d\sigma}&\leq \sqrt{\int_{\{d(x,\partial U_{\infty}) > R_0\}\cap U_{\infty}\setminus K}{e^{-2V(x)}dx}},
\end{align*}
which is smaller than $\epsilon/2$ for a good choice of $K$, say $K'_{\epsilon}$, since the integral at the right-hand side is finite by Hypothesis \ref{H3}.
On the other hand
\begin{align}
\label{s21_e10}
 \int_{\{d(x,\partial U_{m}) \leq R_0\}\cap U_m \setminus K}{\eta_{m}d\sigma}&\leq
             e^{C/2} e^{\lambda_{m}} \kappa
               \int_{\{d(x,\partial U_{\infty}) \leq R_0\}\cap U_{\infty}\setminus K}{ \left( \int_{U_{\infty}}{p_1^{U_{\infty}}(x,y)}dy \right) dx},
\end{align}
where $\kappa=\sup_{m\geq 0}\|\eta_me^{-V}\|_{\infty}<\infty$ thanks to \cite{Cattiaux2008}, and $\lambda_m\leq \lambda_{\infty}$ for all $m\geq 0$. But the integral on the right-hand side is well defined by Hypothesis \ref{H3}, then one can find a compact set $K''_{\epsilon}$ such that \eqref{s21_e10} is bounded by $\epsilon/2$. 
We set $K_{\epsilon}=K'_{\epsilon}\cup K''_{\epsilon}$ so that \eqref{s21_e11} is fulfilled. Since inequality \eqref{s21_e11} occurs for all $\epsilon>0$, the family $(\eta_{m}d\sigma)_{m\geq 0}$ is uniformly tight. Moreover, $\eta_{m}d\sigma$ has a density with respect to the Lebesgue measure, which is bounded by $\kappa e^{-V}$, uniformly in $m\geq 0$. Then it is uniformly bounded on every compact set, so that every limiting distribution is absolutely continuous with respect to the Lebesgue measure.

Let $\eta d\sigma$ be a limiting measure of $(\eta_{m}d\sigma)_{m\geq 0}$. For all $\phi\in  C_0^{\infty}(U_{\infty},\mathds{R})$, the support of $\phi$ belongs to $U_m$ for $m$ big enough, then
\begin{align*}
\left\langle \eta, {\cal L}^{\infty}\phi \right\rangle_{\sigma,\infty}&=\lim_{m\rightarrow\infty} \left\langle \eta_{m}, {\cal L}^{m}\phi \right\rangle_{\sigma,m}\\
                                             &=\lim_{m\rightarrow\infty} \left\langle {\cal L}^{m*} \eta_{m}, \phi \right\rangle_{\sigma,m}\\
                                             &=\lim_{m\rightarrow\infty} -\lambda_{m}\left\langle \eta_{m}, \phi \right\rangle_{\sigma,m}\\
                                             &=-\lambda_{\infty} \left\langle \eta, \phi \right\rangle_{\sigma,\infty}.
\end{align*}
Thanks to the elliptic regularity Theorem, $\eta$ is of class $C^2$ and fulfills ${\cal L}^{\infty*} \eta=-\lambda_{\infty} \eta$. But the eigenvalue $\lambda_{\infty}$ is simple, then $\eta$ is proportional to $\eta_{\infty}$. Let $\beta\geq 0$ be the non-negative constant such that $\eta=\beta \eta_{\infty}$. In particular, there exists an increasing function $\phi:\mathds{N}\mapsto\mathds{N}$ such that 
$\eta_{\phi(m)}d\sigma$ converges weakly to $\beta\eta_{\infty}d\sigma$.

Let us prove that $\beta$ is positive. For all compact subset $K\subset U_{\infty}$, we have
\begin{align}
\beta\left\langle \eta_{\infty},e^V \mathds{1}_K\right\rangle_{\sigma,\infty}&=\lim_{m\rightarrow \infty} \left\langle \eta_{{\phi(m)}},\mathds{1}_K e^V\right\rangle_{\sigma,{\phi(m)}}\nonumber\\
                                                    &\geq \lim_{m\rightarrow \infty} \frac{1}{\kappa} \left\langle \eta_{{\phi(m)}},\mathds{1}_K\eta_{\phi(m)}\right\rangle_{\sigma,{\phi(m)}}\nonumber\\
                                                    &\geq \frac{1}{\kappa} \left(1-\sup_{m\geq 0}\left\langle \eta_{m},\mathds{1}_{U_m\setminus K}\eta_m\right\rangle_{\sigma,m}\right),\label{ceq5}
\end{align}
where $\kappa=\sup_{m\geq 0}\|\eta_me^{-V}\|_{\infty}<\infty$. For all $m\geq 0$ and all $R>0$,
\begin{align}
\label{ceq4}
\left\langle \eta_{m},\mathds{1}_{U_m\setminus K}\eta_m\right\rangle_{\sigma,m}\leq \frac{1}{\overline{G}(R)}\left\langle\eta_m,\mathds{1}_{|x|\geq R}G\eta_m \right\rangle_{\sigma,m}+\left\langle\eta_m,\mathds{1}_{\{|x|<R\}\setminus K}\eta_{m}\right\rangle_{\sigma,m},
\end{align}
where $G$ and $\overline{G}$ are defined in Hypothesis \ref{H3}.
Let us prove that $\left\langle\eta_m,G\eta_m \right\rangle_{\sigma,m}$ is uniformly bounded in $m\geq 0$.
For all $x\in U_m$, \eqref{eq23} leads to
\begin{equation*}
\frac{1}{2}G(x) \eta_m(x)=\lambda_{m} \eta_{m}(x)+ \frac{1}{2}e^{V(x)}\Delta (\eta_m e^{-V})(x).
\end{equation*}
Then
\begin{align*}
 \left\langle\eta_m,G\eta_m \right\rangle_{\sigma,m}&=\lambda_{m} \left\langle\eta_m,\eta_m \right\rangle_{\sigma,m}+ \frac{1}{2}\int_{U_m}{\eta_m(x)e^{-V(x)}\Delta (\eta_m e^{-V})(x)dx}\nonumber\\
                                                 &=\lambda_{m} - \int_{U_m}{ |\nabla \eta_{m}(x)e^{-V(x)}|^2 dx}\nonumber\\
                                                 &\leq \lambda_{1},\nonumber
\end{align*}
where the second equality is a consequence of the Green's formula (see \cite[Corollary 3.2.4]{Allaire2005}).
But $\overline{G}(R)$ goes to $+\infty$ when $R\rightarrow\infty$, then there exists $R_1>0$ such that $\frac{1}{\overline{G}(R_1)}\left\langle\eta_m,\mathds{1}_{|x|\geq R_1}G\eta_m \right\rangle_{\sigma,m}\leq \frac{1}{4}$. Since $\kappa=\sup_{m\geq 0}\|\eta_me^{-V}\|_{\infty}<\infty$, we deduce from \eqref{ceq4} that
\begin{equation*}
\left\langle \eta_{m},\mathds{1}_{U_m\setminus K}\eta_m\right\rangle_{\sigma,m}\leq \frac{1}{4}+\kappa^2 \int_{U_{\infty}}\mathds{1}_{\{|x|<R_1\}\setminus K} dx.
\end{equation*}
But one can find a compact subset $K_1\subset U_{\infty}$ such that $\int_{U_{\infty}}\mathds{1}_{\{|x|<R_1\}\setminus K_1} dx\leq \frac{1}{4\kappa^2}$, then we have from \eqref{ceq5}
\begin{align*}
\beta\left\langle \eta_0,\mathds{1}_{K}\right\rangle_{\sigma}&\geq \frac{1}{2\kappa}.
\end{align*}
It yields that $\beta>0$ and Proposition \ref{pr2} follows.
\end{proof}

\subsection{Numerical simulations}
\label{simulation}
\subsubsection{The Wright-Fisher case}

The Wright-Fisher with values in $]0,1[$ conditioned to be absorbed at $0$ is the diffusion process driven by the SDE 
\begin{equation*}
 dZ_t=\sqrt{Z_t(1-Z_t)}dB_t-Z_t dt,\ Z_0=z\in]0,1[,
\end{equation*}
and absorbed when it hits $0$ ($1$ is never reached). Huillet proved in \cite{Huillet2007} that the Yaglom limit of this process exists and has the density $2-2x$ with respect to the Lebesgue measure.
In order to apply Theorem \ref{th3_1}, we define
 $\mathds{P}^{\infty}$ as the law of $X^{\infty}_.=\arccos(1-2 Z_.)$. Then $\mathds{P}^{\infty}$ is the law of the diffusion process with values in $U_{\infty}=]0,\pi[$, driven by the SDE
\begin{equation*}
 dX^{\infty}_t=dB_t-\frac{1-2\cos X^{\infty}_t}{2\sin X^{\infty}_t } dt,\ X^{\infty}_0=x\in]0,\pi[,
\end{equation*}
absorbed when it hits $0$ ($\pi$ is never reached). One can easily check that this diffusion process fulfills Hypothesis \ref{H3}.
We denote by $\nu_{{\infty}}$ its Yaglom limit.

For all $m\geq 1$, we define $U_{m}=]\frac{1}{m},\pi-\frac{1}{m}[$. Let $\mathds{P}^{m}$ and $\nu_{m}$ be as in Section \ref{partApproximation}.
We proceed to the numerical simulation of the $N$-interacting particle system $(X^{m,1},...,X^{m,N})$ with $m=1000$ and $N=1000$. This leads us to the computation of $E({\cal X}^{m,N})$, which is an approximation of $\nu_{\infty}$. After the change of variable $Z.=2\cos(X.)$, we see on Figure \ref{fig4} that the simulation is very close to the expected result $(2-2x)dx$, which shows the efficiency of the method.

\begin{figure}[htbp]
\begin{center}
\input{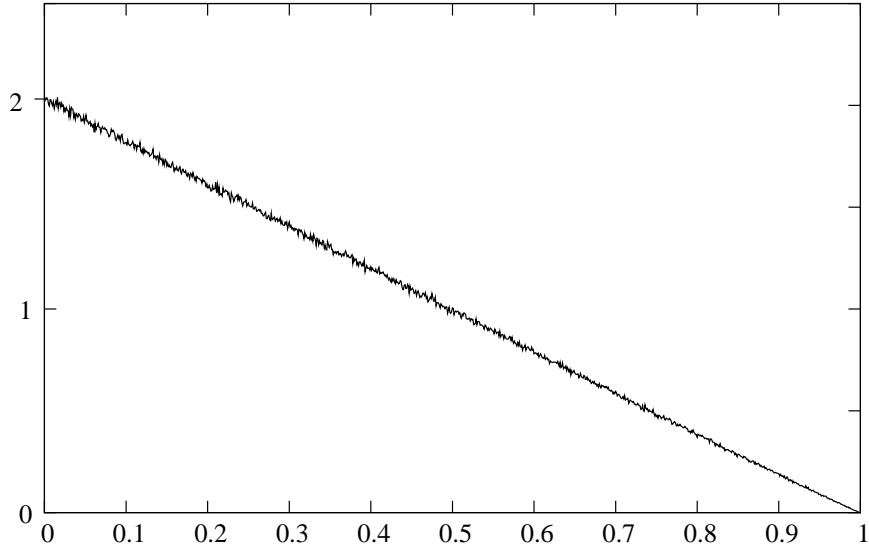}
\caption{$E({\cal X}^{m,N})$ in the Wright-Fisher case}
\label{fig4}
\end{center}
\end{figure}

\subsubsection{The logistic case}
\label{par5}
The logistic Feller diffusion with values in $]0,+\infty[$ is defined by the stochastic differential equation
\begin{equation}
 \label{eq77}
 dZ_t=\sqrt{Z_t}dB_t+(rZ_t-cZ_t^2)dt,\ Z_0=z>0,
\end{equation}
and absorbed when it hits $0$. Here $B$ is a $1$-dimensional Brownian motion and $r,c$ are two positive constants.
In order to use Theorem \ref{th3_1}, we make the change of variable $X.=2\sqrt{Z.}$. This leads us to the study of the diffusion process with values in $U_{\infty}=]0,+\infty[$, which is absorbed at $0$ and satisfies the SDE
\begin{equation*}
 dX^{\infty}_t=dB_t-\left(\frac{1}{2X^{\infty}_t}-\frac{rX^{\infty}_t}{2}+\frac{c(X^{\infty}_t)^3}{4}\right)dt,\ X^{\infty}_0=x\in]0,+\infty[.
\end{equation*}
We denote by $\mathds{P}^{\infty}$ its law. Cattiaux et al. proved in \cite{Cattiaux2009} that Hypothesis \ref{H3} is fulfilled in this case. Then
the Yaglom limit $\nu_{\infty}$ associated with $\mathds{P}^{\infty}$ exists and one can apply Theorem \ref{th3_1} with $U_m=]\frac{1}{m},m[$ for all $m\geq1$. For all $N\geq 2$, we denote by $\mathds{P}^{m}$ the law of the diffusion process restricted to $U_m$ and by ${\cal X}^{m,N}$ the empirical stationary distribution of the $N$-interacting particle process associated with $\mathds{P}^{m}$.

We've proceeded to the numerical simulation of the interacting particle process for a large number of particles and a large value of $m$. 
This allows us to compute $E({\cal X}^{m,N})$, which gives us a numerical approximation of $\nu_{\infty}$, thanks to Theorem \ref{th3_1}.

In the numerical simulations below, we set $m=10000$ and $N=10000$. We compute $E({\cal X}^{m,N})$ for different values of the parameters $r$ and $c$ in \eqref{eq77}. The results are graphically represented in Figure \ref{fig10}.
As it could be wanted for, greater is $c$, closer is the support of the QSD to $0$. We thus numerically describe the impact of the linear and quadratic terms on the Yaglom limit.

\begin{figure}[htbp]
\begin{center}
 \input{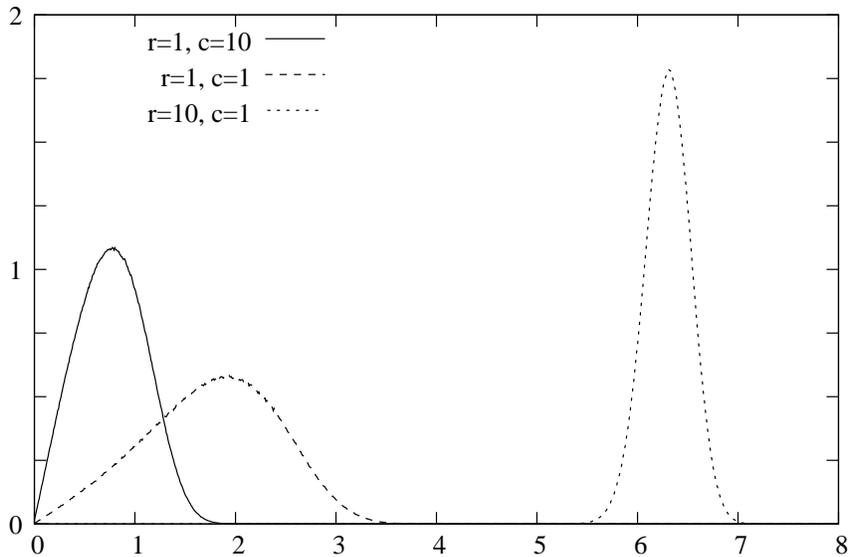}
\caption{$E({\cal X}^{m,N})$ for the diffusion process \eqref{eq77}, with different values of $r$ and $c$}
\label{fig10}
\end{center}
\end{figure}

\subsubsection{Stochastic Lotka-Volterra Model}
\label{partApproximation_3}
We apply our results to the stochastic Lotka-Volterra system with values in $D=\mathds{R}_+^2$ studied in \cite{Cattiaux2008}, which is defined by the following stochastic differential system
\begin{equation*}
 \begin{split}
  dZ^1_t=\sqrt{\gamma_1 Z_t^1}dB^1_t+\left( r_1 Z^1_t - c_{11}(Z^1_t)^2-c_{12}Z^1_t Z^2_t \right) dt,\\
  dZ^2_t=\sqrt{\gamma_2 Z_t^2}dB^2_t+\left( r_2 Z^2_t -c_{21}Z^1_t Z^2_t - c_{22}(Z^2_t)^2 \right) dt,
 \end{split}
\end{equation*}
where $(B^1,B^2)$ is a bi-dimensional Brownian motion. We are interested in the process absorbed at $\partial D$.

More precisely, we study the process $X^{\infty}=(Y^1,Y^2)=(2\sqrt{Z^1_./\gamma_1},2\sqrt{Z^2_./\gamma_2})$, with values in $U_{\infty}=\mathds{R}_+^2$, which satisfies the SDE \eqref{eqr1} and is absorbed at $\partial U_{\infty}$. We denote its law by $\mathds{P}^{\infty}$. The coefficients are supposed to satisfy
\begin{equation}
  \label{eql3}
  c_{11},c_{21}>0,\ c_{12}\gamma_2=c_{21}\gamma_1<0\ \text{and}\ c_{11}c_{22}-c_{12}c_{21}>0.
\end{equation}
In \cite{Cattiaux2008}, this case was called \emph{the weak cooperative case} and the authors proved that it is a sufficient condition for Hypothesis \ref{H3} to be fulfilled. Then the Yaglom limit
$\nu_{\infty}=\lim_{t\rightarrow+\infty}\mathds{P}^{\infty}_{x}\left(X^{\infty}\in . | t<\tau_{\partial}\right)$
is well defined and we are allowed to apply Theorem \ref{th3_1}. For each $m\geq 1$, we define $U_m$ as it is described on Figure \ref{fig2}.  With this definition, it is clear that all conditions of Theorems \ref{thA_1} and \ref{th3_1} are fulfilled.

\begin{figure}[htbp]
\begin{center}
\includegraphics[height=8cm]{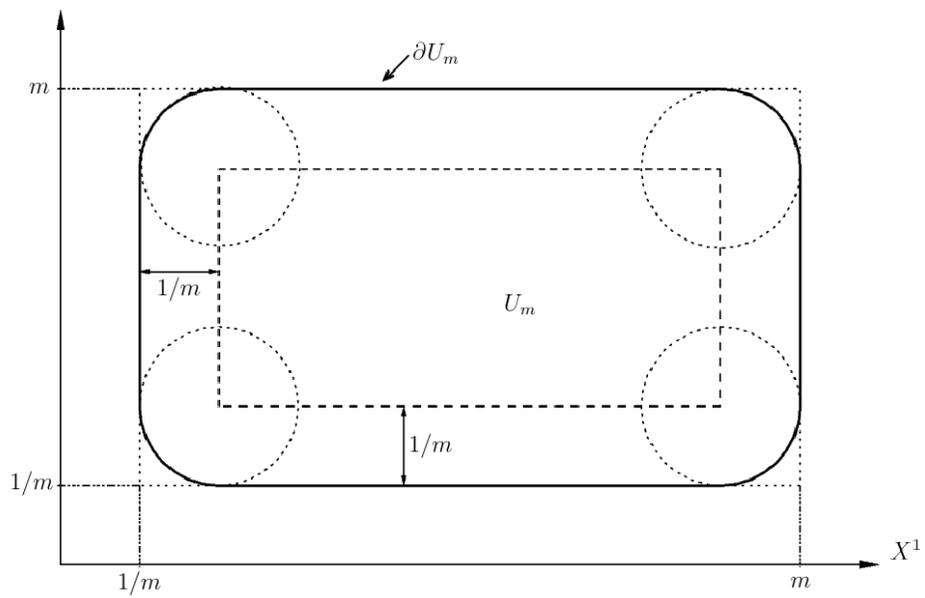}
\caption{Definition of $U_m$} 
\label{fig2}
\end{center}
\end{figure}

We choose $m=10000$ and we simulate the long time behavior of the interacting particle process with $N=10000$ particles for different values of $c_{12}$ and $c_{21}$. The values of the other parameters are $r_1=1=r_2=1,\ c_{11}=c_{22}=1,\ \gamma_1=\gamma_2=1$.
The results are illustrated on Figure \ref{figa1}.
One can observe that a greater value of the cooperating coefficients $-c_{12}=-c_{21}$ leads to a Yaglom limit whose support is further from the boundary and covers a smaller area. In other words, the more the two populations cooperate, the bigger the surviving populations are.

\begin{figure}[htbp]
\begin{center}
 \includegraphics[height=12cm]{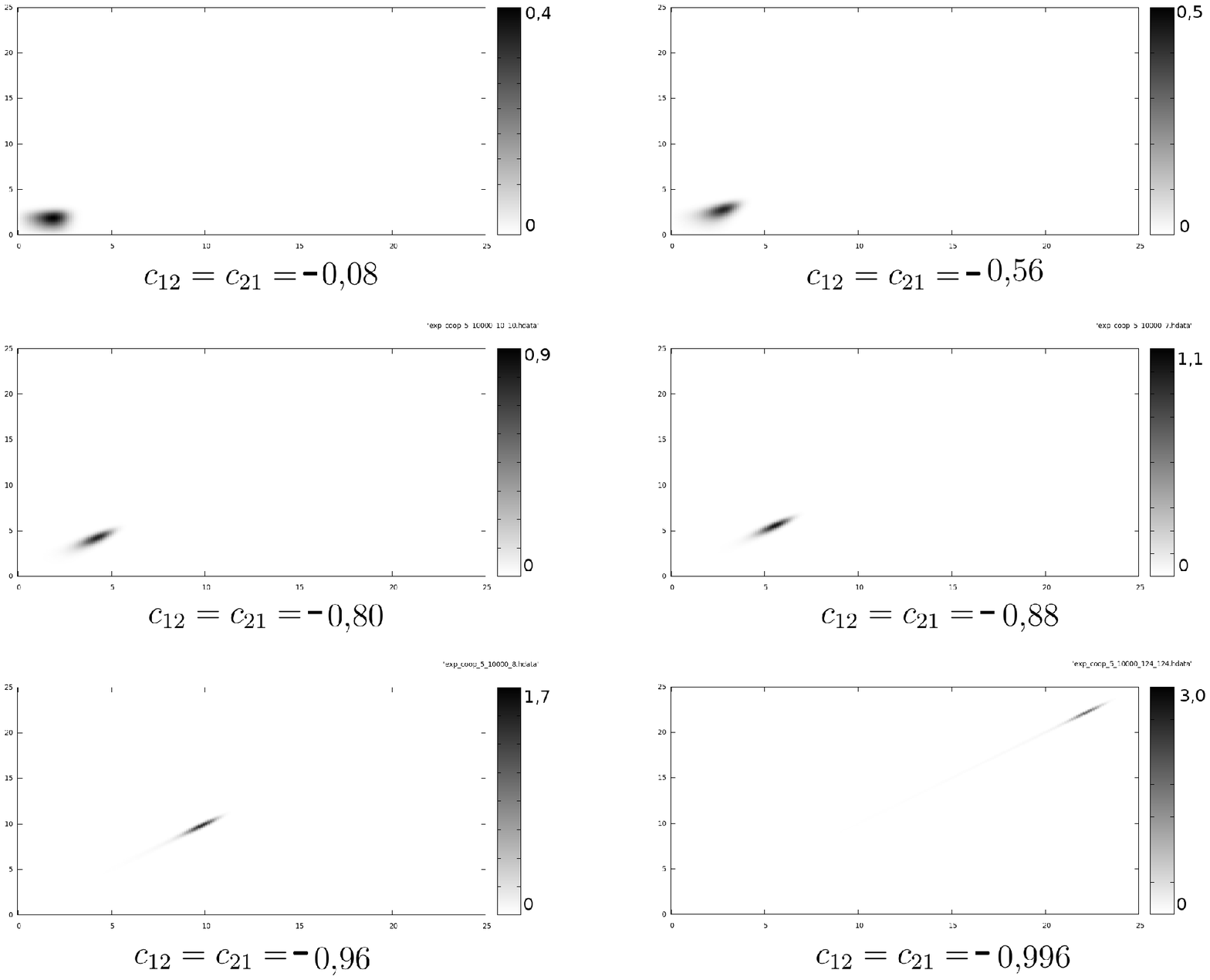}
 \caption{Empirical stationary distribution of the interacting particle process for different values of $c_{12}=c_{21}$}
 \label{figa1}
\end{center}
\end{figure}

\paragraph{Acknowledgments} I am extremely grateful to my Ph.D. supervisor Sylvie M\'el\'eard for his careful and indispensable help on the form and the content of this paper. I would like to thank Pierre Collet and my colleagues Jean-Baptiste Bellet and Khalid Jalalzai for their advices on functional analysis.

\clearpage

 \nocite{Cavender1978,Collet1995,Darroch1965,Ferrari1995,Martinez1998,Yaglom1947}

 \bibliography{./bibliography/compilation}
 \bibliographystyle{abbrv}
\end{document}